\documentclass[11pt]{amsart}
\usepackage{amsmath, amsthm, amscd, amsfonts, amssymb, graphicx, color}
\usepackage[bookmarksnumbered, colorlinks, plainpages]{hyperref}
\hypersetup{colorlinks=true,linkcolor=red, anchorcolor=green, citecolor=cyan, urlcolor=red, filecolor=magenta, pdftoolbar=true}
\usepackage{tikz}

\newtheorem{thm}{Theorem}[section]
\newtheorem{lemma}[thm]{Lemma}
\newtheorem{prop}[thm]{Proposition}
\newtheorem{cor}[thm]{Corollary}
\theoremstyle{definition}
\newtheorem{defn}[thm]{Definition}
\newtheorem{eg}[thm]{Example}

\theoremstyle{definition}
\newtheorem{rmk}[thm]{Remark}
\theoremstyle{definition}
\newtheorem{note}[thm]{Note}
\theoremstyle{definition}
\newtheorem{question}[thm]{Question}
\numberwithin{equation}{section}

\DeclareRobustCommand{\rchi}{{\mathpalette\irchi\relax}}
\newcommand{\irchi}[2]{\raisebox{\depth}{$#1\chi$}}
\DeclareRobustCommand{\rrho}{{\mathpalette\irrho\relax}}
\newcommand{\irrho}[2]{\raisebox{\depth}{$#1\rho$}}

\begin{document}
\title[Strongly irreducible factorization]{Strongly irreducible factorization of quaternionic operators and Riesz decomposition theorem}
\author{P. Santhosh Kumar}
\address{Indian Statistical Institute Bangalore, Statistics and Mathematics Unit, 8th Mile, Mysore Road, 
	Bengaluru 560 059, India}
\email{santhosh.uohmath@gmail.com}
\begin{abstract}
    Let $T$ be a bounded quaternionic normal operator on a right quaternionic Hilbert space $\mathcal{H}$. We show that $T$ can be factorized in a strongly irreducible sense, that is, for any $\delta >0$ there exist a compact operator $K$ with $\|K\|< \delta$, a partial isometry $W$ and a strongly irreducible operator $S$ on $\mathcal{H}$ such that 
		\begin{equation*}
		T = (W+K) S.
		\end{equation*}
	We illustrate our result with an example. We also prove a quaternionic version of the Riesz decomposition theorem and as a consequence, show that if the spherical spectrum of a bounded quaternionic operator (need not be normal) is disconnected by a pair of  disjoint axially symmetric closed subsets, then it  is strongly reducible.
\end{abstract}
\maketitle
\section{Introduction}  According to the Frobenius theorem for real division algebras, the algebra of Hamilton quaternions \cite{Hamilton} is the only finite dimensional associative division algebra that contains $\mathbb{R}$ and $\mathbb{C}$ as proper real subalgebras. Similar to the case of real and complex matrices, the theory of matrices over the real algebra of quaternions  has been well developed (see \cite{Adler, Finkelstein, Weigmann, Zhang} and references therein) in the literature. For example, some of the fundamental results like Schur's canonical form and Jordan canonical form are extended to matrices with quaternion entries \cite{Zhang}.  In particular,   the  Jordan canonical form shows that every square matrix over quaternions can be reduced under similarity to a direct sum of Jordan blocks.  The Jordan canonical form determines its complete similarity invariants and establishes the structure of quaternionic operator on finite dimensional right quaternionic Hilbert spaces. Another significant aspect in this direction is the diagonalization, through which the matrix is factorized interms of its restriction to eigenspaces. In particular, as in the case of complex matrices, every quaternionic normal matrix is  diagonalizable. It is worth mentioning the result due to  Weigmann \cite{Weigmann} that a quaternion matrix is normal if and only if its adjoint is a polynomial of given matrix with real coefficients. In this article, we concentrate on the class of normal operators on right quaternionic Hilbert spaces and their factorization in a strongly irreducible sense. For this, we adopt the notion of strong irreducibility  proposed by F. Gilfeather \cite{Gilfeather} and Z. J. Jiang \cite{Jiang}  to the class of quaternionic operators in order to replace the notion of  Jordan block for infinite dimensional right quaternionic Hilbert spaces.


On the other hand, the study of quaternionic normal operators gained much attention in the form of spectral theory for quantum theories and various versions of quaternionic functional calculus  (see \cite{Alpay-R, Colombo, Ghiloni} for details). As far as factorization of quaternionic operators concern, the well known polar decomposition theorem shows that every bounded or densely defined closed quaternionic operator can be decomposed as a product of a partial isometry and a positive operator \cite{Ghiloni, GandP}. Furthermore, the authors of \cite{GandP} provided a necessary and sufficient condition for any arbitrary factorization of densely defined closed quaternionic operator to be a polar decomposition. Note that the positive operator that appear in the polar decomposition is the modulus of the given quaternionic operator which has several reducing subsapces as in the complex case  \cite{Conway}. In summary, the polar decomposition establishes the factorization of quaternionic operator as a product of a partial isometry and an operator having several reducing subspaces.  A natural question that arises from  previous observation is the following:  ``whether a given quaternionic operator be decomposed as a product of a partial isometry and an operator with no reducing subspace?"  We answer this question for  quaternionic normal operators by proving a factorization in a strongly irreducible sense, by means of decomposing the given operator as a product of a sufficiently small  compact perturbation of a partial isometry and a strongly irreducible quaternionic operator. Our result is the quaternionic version (for normal operators) of the result proved recently in \cite{Tian} by G. Tian et al. In which, the authors employed the properties of Cowen-Douglas operators related to complex geometry on complex Hilbert spaces. Also see \cite{Luo} for such factorization in finite dimensional Hilbert spaces.

We organize this article in four sections. In the second section, we give some basic definitions and results that are useful for proving  our assertions. In the third section,  we prove factorization of quaternionic normal operator in a strongly irreducible sense. The final section is dedicated to Riesz decomposition theorem for quternionic operators and  provide a suffient condition for strong irreducibility.  


\vspace{0.5cm}
\section{Preliminaries}

An Irish mathematician Sir William Rowan Hamilton \cite{Hamilton} described this number system known as ``quaternions" as an extension of complex numbers. A quaternion is of the form:
\begin{equation*}
q = q_{0}+q_{1}i+q_{2}j+q_{3}k,
\end{equation*} 
where $q_{\ell} \in \mathbb{R}$ for $\ell = 0,1,2,3$ and $i,j,k$ are the fundemental quaternion units satisfying,
\begin{equation}\label{eq:multiplication}
i^{2}= j^{2} = k^{2}= -1 =  i\cdot j \cdot k.
\end{equation}
Note that the collection of all quaternions is denoted by $\mathbb{H}$ and it is a non-commutative real division algebra equipped with the usual vector space operations addition and scalar multiplication same as in the complex field $\mathbb{C}$, and with the ring multiplication  given by Equation (\ref{eq:multiplication}).

For every $q\in \mathbb{H}$, the real and the imaginary part of `$q$' is defined as $\text{re}(q) = q_{0}$ and $\text{im}(q) =q_{1}i+q_{2}j+q_{3}k$ respectively. Then the conjugate and the modulus of  $q \in \mathbb{H}$  is given respectively by 
\begin{equation*}
\overline{q} = q_{0}-(q_{1}i+q_{2}j+q_{3}k) \; \text{and}\;  |q|= \sqrt{q_{0}^{2}+q_{1}^{2}+q_{2}^{2}+q_{3}^{2}}.
\end{equation*} 
The set of all imaginary unit quaternions in $\mathbb{H}$ denoted by $\mathbb{S}$ and it is defined as
\begin{equation*}
\mathbb{S}:= \big\{q\in \mathbb{H}:\; \overline{q}=-q\; \&\; |q|=1\big\}= \big\{q\in \mathbb{H}:\; q^{2}=-1\big\}.
\end{equation*}
However $\mathbb{S}$ is a $2$- dimensional sphere in $\mathbb{R}^{4}$.
The real subalgebra of $\mathbb{H}$ generated by $\{1, m\}$, where $m\in \mathbb{S}$, is denoted by $\mathbb{C}_{m}: = \{\alpha + m\beta:\; \alpha, \beta \in \mathbb{R}\}$ called a slice of $\mathbb{H}$. It is immediate to see that, for every  $q \in \mathbb{H}$, there is a unique $m_{q}:= \frac{\text{im}(q)}{|\text{im}(q)|} \in \mathbb{S}$ such that 
\begin{equation*}
q = \text{re}(q) + m_{q} |\text{im}(q)|\; \in \mathbb{C}_{m_{q}} . 
\end{equation*} 
Consequently, $\mathbb{H}= \bigcup\limits_{m\in \mathbb{S}} \mathbb{C}_{m}$ and  $\mathbb{C}_{m} \cap \mathbb{C}_{n} = \mathbb{R}$, for all  $m \neq \pm n \in \mathbb{S}$. There is an equivalence relation on $\mathbb{H}$ given by 
\begin{equation*}
p \sim q  \Leftrightarrow \; p = s^{-1}qs, \; \text{for some}\; s\in \mathbb{H}\setminus \{0\}.
\end{equation*}
For every $q \in \mathbb{H}$, the equivalence class of $q$ is expressed as,
\begin{equation}\label{Equation: equivalenceclass}
[q] = \big\{p \in \mathbb{H}:\;  \text{re}(p) = \text{re}(q)\;\text{and}\; |\text{im}(p)|= |\text{im}(q)|\big\}. 
\end{equation}
\begin{defn}\cite{Ghiloni}\label{Definition: circularization}
	Let $S$ be a non-empty subset of $\mathbb{C}$. 
	\begin{enumerate}
	\item If $S$ is invariant under complex conjugation, then  the {\it circularization}  $\Omega_{S}$ of ${S}$ is defined by
	\begin{equation*}
	\Omega_{S} = \big\{\alpha + m \beta :\; \alpha, \beta \in \mathbb{R}, \alpha + i \beta \in S, m \in \mathbb{S}\big\}.
	\end{equation*}
        \item a subset $K$ of $\mathbb{H}$ is called {\it circular} or \emph{axially symmetric}, if $K = \Omega_{S}$, for some $S \subseteq \mathbb{C}$ which is invariant under complex conjugation. 	
        \end{enumerate}
        \end{defn}
       Note that by Definition \ref{Definition: circularization} and Equation (\ref{Equation: equivalenceclass}), we can express $\Omega_{S}$ as the union of the equivalence class of its elements,
       \begin{equation*}
	\Omega_{S} = \bigcup\limits_{z \in S}[z]. 
	\end{equation*}
It follows that the closure of $\Omega_{S}$ is the circularization of the closure of $S$. That is,
\begin{equation}\label{Equation: closureofOmegaS}
    \overline{\Omega}_{S} = \Omega_{\overline{S}}.
\end{equation}

The properties  described above are useful for later sections. A more detailed discussion about quaternions can be found in \cite{Adler,Colombo,Ghiloni,Zhang}.  Now we discuss some basic definitions and existing results related to quaternionic Hilbert spaces.

\begin{defn}\cite[section 2.2]{Ghiloni} \label{Definition: quaternionicHilbertspace}
	An  \emph{inner product} on a right $\mathbb{H}$-module $\mathcal{H}$ is a map 
	\begin{equation*}
	\langle \cdot, \cdot \rangle\colon \mathcal{H}\times \mathcal{H} \to \mathbb{H}
	\end{equation*}  
	satisfy the following properties,  
	\begin{enumerate}
		\item {\bf Positivity}: $\langle x,x\rangle \geq 0$ for all $x \in \mathcal{H}$. In particular,
		\begin{equation*}
		\langle x, x\rangle = 0 \;\; \text{if and only if }\;\; x = 0. 
		\end{equation*} 
		\item {\bf Right linearity}: $\langle x, yq+z\rangle = \langle x, y \rangle\; q + \langle x, z\rangle $, for all $x,y,z \in \mathcal{H}$, $q \in \mathbb{H}$.
		\item {\bf Quaternionic Hermiticity}: $\langle x,y\rangle = \overline{\langle y, x\rangle}$, for all $x,y \in \mathcal{H}$.
	\end{enumerate}
	The pair $(\mathcal{H}, \langle \cdot, \cdot \rangle)$ is called \emph{quaternionic pre-Hilbert space}. Moreover, $\mathcal{H}$ is said to be 
	\begin{enumerate}
		\item[(i)]\emph{quaternionic Hilbet space}, if $\mathcal{H}$ is complete with respect to the norm induced from the inner product $\langle \cdot, \cdot\rangle$, which is defined by
		\begin{equation*}
		\|x\|= \sqrt{\langle x, x\rangle}, \; \text{for all}\; x \in \mathcal{H}.
		\end{equation*}
		\item[(ii)] \emph{separable}, if $\mathcal{H}$ has a countable dense subset.
	\end{enumerate}
Furthermore, for any subset $\mathcal{N}$ of $\mathcal{H}$, the \emph{span} of $\mathcal{N}$ is defined as
		\begin{equation*}
		span\ \mathcal{N}:= \Big\{ \sum\limits_{\ell = 1}^{n} x_{\ell} q_{\ell}; \; x_{\ell} \in \mathcal{N}, n \in \mathbb{N} \Big\}.
		\end{equation*}
The orthogonal complement of a subspace $\mathcal{M}$ of $\mathcal{H}$ is, 
\begin{equation*}
\mathcal{M}^{\bot} = \big\{ x \in \mathcal{H};\; \langle x, y\rangle = 0, \; \text{for every}\; y \in \mathcal{M}\big\}.
\end{equation*}
\end{defn}

\begin{note}
	The inner product $\langle \cdot, \cdot \rangle $ defined on $\mathcal{H}$ satisfies Cauchy-Schwarz inequality:
	\begin{equation*}
	|\langle x,y\rangle|^{2} \leq \langle x,x\rangle \langle y,y\rangle, \; \text{for all}\; x,y \in \mathcal{H}.
	\end{equation*}
\end{note}
\begin{defn}\cite{Ghiloni}
	Let $\big(\mathcal{H}, \langle \cdot, \cdot \rangle\big)$ be a quaternionic Hilbert space. A subset $\mathcal{N}$ of $\mathcal{H}$ with the property that 
	\[
	\langle z, z^{\prime}\rangle = \left\{ \begin{array}{cc}
	1, & \mbox{if $z = z^{\prime}$}  \\
	0,& \mbox{if $z \neq z^{\prime}$} 
	\end{array} \right.
	\]
	is said to be  \emph{Hilbert basis} if for every $x,y\in \mathcal{H}$, the series $\sum\limits_{z \in \mathcal{N}} \langle x, z \rangle \langle z, y \rangle$ converges absolutely and it holds:
	\begin{equation*}
	\langle x, y \rangle = \sum\limits_{z \in \mathcal{N}} \langle x, z \rangle \langle z, y \rangle.
	\end{equation*}
	Equivalentely, $\overline{span}\ \mathcal{N}= \mathcal{H}$.
\end{defn}
\begin{rmk}
By \cite[Proposition 2.6]{Ghiloni}, every quaternionic Hilbert space $\mathcal{H}$ has a Hilbert basis $\mathcal{N}$. For every $x \in \mathcal{H}$, it is uniquely decomposed as, 
\begin{equation*}
x= \sum\limits_{z \in \mathcal{N}} z \langle z, x\rangle.
\end{equation*} 
\end{rmk}
\begin{eg}\cite{GandP-MultiplicationForm}
	Let $\mu$ be a lebesgue measure on $[0,1]$ and the set of all $\mathbb{H}$-valued square integrable $\mu$-measurable functions on $[0, 1]$ is defined as,
	\begin{equation*}
	L^{2}\big( [0,1]; \mathbb{H}; \mu\big) = \Big\{f\colon [0,1] \to \mathbb{H}\ ;\; \int\limits_{0}^{1}|f(x)|^{2}\; dx < \infty \Big\}.
	\end{equation*}
	It is a right quaternionic Hilbert space with respect to the inner product given by 
	\begin{equation*}
	\langle f, g\rangle = \int\limits_{0}^{1}\overline{f(x)}g(x)\; dx, \; \text{for all}\; f,g \in L^{2}\big( [0,1]; \mathbb{H}; \mu\big).
	\end{equation*}
	Let us fix $m \in \mathbb{S}$. If we define $ e_{r}^{(m)}(x) = \exp\{2\pi m rx\}$, \text{for all} $x \in [0,1]$ and $r \in \mathbb{Z}$. Then the set  $\mathcal{N}= \{e_{r}^{(m)}:\; r \in \mathbb{Z}\}$ is an orthornormal system in $L^{2}\big( [0,1]; \mathbb{H}; \mu\big)$. Further, by the Stone-Weierstrass theorem $\mathcal{N}$ is a Hilbert basis in $L^{2}\big( [0,1]; \mathbb{H}; \mu\big)$. Since $\mathcal{N}$ is countable, we conclude that $L^{2}\big( [0,1]; \mathbb{H}; \mu\big)$ is a separable quaternionic Hilbert space. 
\end{eg}

\begin{defn}
	Let $\mathcal{H}$ be a quaternionic Hilbert space. A map $T \colon \mathcal{H} \to \mathcal{H}$ is said to be {\it right $\mathbb{H}$- linear} or {\it quaternionic operator}, if 
	\begin{equation*}
	T(x+yq) = T(x) + T(y)\; q,\; \text{for all} \; x \in \mathcal{H}.	
	\end{equation*}
	Moreover, $T$ is said to be {\it bounded} or {\it continuous}, if there exist a $\alpha >0$ such that 
	\begin{equation*}
	\|Tx\| \leq \alpha \|x\|, \; \text{for all}\; x \in \mathcal{H}.
	\end{equation*}
	We denote the class of all bounded operators on $\mathcal{H}$ by $\mathcal{B}(\mathcal{H})$ and it is a real Banach algebra with respect to the operator norm defined by
	\begin{equation*}
	\|T\|= \sup\big\{\|Tx\|:\; x\in \mathcal{H},\; \|x\|\leq 1\big\}.
	\end{equation*}
\end{defn}
For every $T \in \mathcal{B}(\mathcal{H})$, by the quaternionic version of the Riesz representation theorem \cite[Theorem 2.8]{Ghiloni}, there exists a unique operator denoted by $T^{\ast} \in \mathcal{B}(\mathcal{H})$, called the {\it adjoint} of $T$ satisfying,
\begin{equation*}
\langle x, Ty\rangle = \langle T^{\ast}x, y\rangle, \; \text{for all}\; x,y \in \mathcal{H}. 
\end{equation*}
If $T \in \mathcal{B}(H)$, then the null space of $T$ is defined by $N(T) = \{x\in \mathcal{H}:\; Tx=0\}$ and the range space of $T$ is defined by $R(T) = \{Tx:\; x\in \mathcal{H}\}$. A closed subspace $\mathcal{M}$ of $\mathcal{H}$ is said to be \emph{invariant} subspace of $T$, if 
\begin{equation*}
    T(x) \in \mathcal{M}, \;\; \text{for every}\;\; x \in \mathcal{M}. 
\end{equation*}
Furthermore, $\mathcal{M}$ is said to be \emph{reducing subspace} of $T$ if $\mathcal{M}$ is an invariant subspace of both $T$ and $T^{\ast}$.

\begin{defn}
	Let $T \in \mathcal{B}(\mathcal{H})$. Then $T$ is said to be
	\begin{enumerate}
		\item \emph{self-adjoint}, if $T^{\ast}=T$,
		\item \emph{anti self-adjoint}, if $T^{\ast} = -T$, 
		\item {\it normal} if $T^{\ast}T = TT^{\ast}$,
		\item \emph{positive}, if $T^{\ast}=T$ and $ \langle x, Tx\rangle \geq 0$ for all $x \in \mathcal{H}$,
		\item \emph{orthogonal projection}, if $T^{\ast}=T$ and $T^{2}=T$,
		\item \emph{partial isometry}, if $\|Tx\| = \|x\|$, for all $x \in N(T)^{\bot}$,
		\item \emph{unitary}, if $T^{\ast}T = TT^{\ast} = I$.
	\end{enumerate}
\end{defn}

Suppose that $T \in \mathcal{B}(\mathcal{H})$ is positive,  then by \cite[Theorem 2.18]{Ghiloni}, there exists a unique positive operator $S \in \mathcal{B}(\mathcal{H})$ such that $S^{2}= T$. Such an operator $S$ is called the \emph{positive square root} of $T$ and it is dentoed by $S := T^{\frac{1}{2}}$. In fact, for every $T \in \mathcal{B}(\mathcal{H})$,  the modulus $|T|$ is defined as the positive square root of $T^{\ast}T$, that is, $|T| := (T^{\ast}T)^{\frac{1}{2}}$.

We know from the well known \emph{Cartesian decomposition} that every bounded normal operator on a complex Hilbert space can be decomposed uniquely as $A + i B$, where $A, B$ are bounded self-adjoint operators. There is a quaternionic analog of this result in which the role of `$i$' is replaced by an anti self-adjoint unitary operator. 
\begin{thm}\cite[Theorem 5.9]{Ghiloni} \label{Theorem: Cartesian}
	Let $T \in \mathcal{B}(\mathcal{H})$ be normal. Then there exists an anti self-adjoint unitary operator $J \in \mathcal{B}(\mathcal{H})$ such that $TJ = JT$, $T^{\ast}J = JT^{\ast}$ and 
	\begin{equation*}
	T = \frac{1}{2}(T+T^{\ast}) +\frac{1}{2}\; J |T-T^{\ast}|.
	\end{equation*}
 Here $J$ is uniquely determined by $T$ on $N(T-T^{\ast})^{\bot}$. Moreover, the operators $(T+T^{\ast}), \;|T-T^{\ast}| $ and $J$ commute mutually. 
\end{thm}
\begin{defn}\cite[Definition 4.8.1]{Colombo}
	For every $T \in \mathcal{B}(\mathcal{H})$ and $q \in \mathbb{H}$, let us associate an operator $\Delta_{q}(T): = T^{2} -2 re(q)T + |q|^{2}I$. Then
	\begin{enumerate}
		\item the {\it spherical spectrum} of $T$ is defined as, 
		\begin{equation*}
		\sigma_{S}(T) = \big\{q\in \mathbb{H}:\; \Delta_{q}(T) \; \text{is not invertible in}\; \mathcal{B}(\mathcal{H})\big\}.
		\end{equation*}
		So the \emph{spherical resolvent} of $T$ denoted by $\rho_{S}(T)$ is the complement of $\sigma_{S}(T)$. That is, 
		\begin{equation*}
		    \rrho_{S}(T):= \mathbb{H}\setminus \sigma_{S}(T).
		\end{equation*}
		\item The \emph{spherical point spectrum} of $T$ is defined by 
		\begin{equation*}
		\sigma_{p^{S}}(T):= \big\{q \in \mathbb{H}:\; N(\Delta_{q}(T)) \neq \{0\}\big\}. 
		\end{equation*}
	\end{enumerate}
	The spherical spectrum $\sigma_{S}(T)$ is a nonempty compact subset of $\mathbb{H}$. 
\end{defn}
Now we recall the notion of slice Hlibert space associated to the given quaternionic Hilbert space $\mathcal{H}$, anti self-adjoint unitary operator $J\in \mathcal{B}(\mathcal{H})$ and $m \in \mathbb{S}$. 
\begin{defn} \cite[Definition 3.6]{Ghiloni}
	Let $m \in \mathbb{S}$ and $J \in \mathcal{B}(\mathcal{H})$ be an anti self-adjoint unitary operator. These subsets $\mathcal{H}^{Jm}_{\pm}$ of $\mathcal{H}$ associated with $J$ and $m$ are defined by setting
	\begin{equation*}
	\mathcal{H}^{Jm}_{\pm}: = \{x\in \mathcal{H}\; :\; J(x) = \pm\; x\cdot m \}.
	\end{equation*}
\end{defn}
\begin{rmk}
	For each $x \in \mathcal{H}$ and $m \in \mathbb{S}$, define $x_{\pm}:= \frac{1}{2}(x\mp Jx\cdot m)$, then 
	\begin{equation}\label{Equation: Jx}
	J(x_{\pm}) = \frac{1}{2}\big(Jx \mp J^{2}x\cdot m\big) = \frac{1}{2}\big(Jx \pm x\cdot m\big) = \pm \; x_{\pm}\cdot m.
	\end{equation}
	It implies that $x_{\pm} \in \mathcal{H}^{Jm}_{\pm}$. Since $x \mapsto \pm\; x\cdot m$ and $x \mapsto Jx$ are continuous, then $\mathcal{H}^{Jm}_{\pm}$ is a non-trivial closed subsets of $\mathcal{H}$.  In fact, we see that the inner product on $\mathcal{H}$ restricted to $\mathcal{H}^{Jm}_{\pm}$ is $\mathbb{C}_{m}$-valued as follows: let $\alpha +m \beta \in \mathbb{C}_{m}$ and $u,v\in \mathcal{H}^{Jm}_{\pm}$ for $m \in \mathbb{S}$, then 
	\begin{align*}
	\langle u, v\rangle (\alpha \pm m \beta) 
	&= \alpha\langle u, v\rangle \pm \beta  \langle u, v\cdot m\rangle\\
	&= \alpha\langle u, v\rangle \pm \beta  \langle u, Jv\rangle\\
	&=  \alpha\langle u, v\rangle \mp \beta  \langle Ju, v\rangle \;\;\;\; (\text{since} \;J^{\ast} = - J)\\
	&=\alpha\langle u, v\rangle \mp  \beta  \langle  u\cdot m, v\rangle\\
	&=(\alpha \pm m \beta )\langle u, v\rangle \; \;\;\;\;\;\;\; (\text{since} \; \overline{m} = -m ). 
	\end{align*}
	Then by being a closed subspace of $\mathcal{H}$, we conclude that $\mathcal{H}^{Jm}_{\pm}$ is a Hilbert space over the field $\mathbb{C}_{m}$. These Hilbert spaces $\mathcal{H}^{Jm}_{\pm}$ are known as \emph{slice Hilbert spaces}. 	
	As a $\mathbb{C}_{m}$- Hilbert space $\mathcal{H}$ has the following decomposition  \cite[Lemma 3.10]{Ghiloni}
	\begin{equation}\label{Equation: directsum}
	\mathcal{H} = \mathcal{H}^{Jm}_{+} \oplus \mathcal{H}^{Jm}_{-}, \; \text{for any}\; m \in \mathbb{S}.
	\end{equation}
	Furthermore, if $\mathcal{N}$ is a Hilbert basis of $\mathcal{H}^{Jm}_{+}$, then $\mathcal{N}\cdot n = \{z\cdot n:\; z\in \mathcal{N}\}$, where $n \in \mathbb{S}$ such that $mn = -nm$, is a Hilbert basis of $\mathcal{H}^{Jm}_{-}$. From Equation (\ref{Equation: directsum}), it follows that $\mathcal{N}$ is also Hilbert basis of $\mathcal{H}$.
\end{rmk}

We denote the class of all bounded $\mathbb{C}_{m}$- linear operators on $\mathcal{H}^{Jm}_{+}$ by $\mathcal{B}(\mathcal{H}^{Jm}_{+})$. The following proposition develop a technique to extend $\mathbb{C}_{m}$- linear operator on $\mathcal{H}^{Jm}_{+}$ (for any $m \in \mathbb{S}$) to the quaternionic operator on $\mathcal{H}$.

\begin{prop}\label{Proposition: Key}\cite{Ghiloni}
	Let $J \in \mathcal{B}(\mathcal{H})$ be anti self-adjoint unitary and $m \in \mathbb{S}$. If  $T\in \mathcal{B}(\mathcal{H}^{Jm}_{+})$, then there exist a unique  quaternionic operator $\widetilde{T} \in \mathcal{B}(\mathcal{H})$ such that $\widetilde{T}(x) = T(x)$, for all $x \in \mathcal{H}^{Jm}_{+}$. The following additional facts hold.
	\begin{enumerate}
		\item $\|\widetilde{T}\| = \|T\|$.
		\item $J\widetilde{T} = \widetilde{T}J$.
		\item $(\widetilde{T})^{\ast} = \widetilde{T^{\ast}}$.
		\item If $S \in \mathcal{B}(\mathcal{H}^{Jm}_{+})$, then $\widetilde{ST} = \widetilde{S} \widetilde{T}$
		\item If $T$ is invertible, then $\widetilde{T}$ is invertible and the inverse is given by
		\begin{equation*}
		(\widetilde{T})^{-1} = \widetilde{{T^{-1}}}.
		\end{equation*}
	\end{enumerate}
	On the other hand, let $V \in \mathcal{B}(\mathcal{H})$, then $V = \widetilde{U}$ for some $U \in \mathcal{B}(\mathcal{H}^{Jm}_{+})$ if and only if $JV = VJ$.
\end{prop}
Precisely, the extension $\widetilde{T}$ of the operator $T$ is defined as,
\begin{align*}
\widetilde{T}(x) = \widetilde{T}(x_{+}+x_{-}) &= \widetilde{T}(x_{+}) + \widetilde{T}(x_{-})\\
&= \widetilde{T}(x_{+}) - \widetilde{T}(x_{-}\cdot n)\cdot n\\
&= T(x_{+}) - T(x_{-}\cdot n)\cdot n,
\end{align*}
for all $x = x_{+} + x_{-} \in \mathcal{H}^{Jm}_{+} \oplus \mathcal{H}^{Jm}_{-}$.

\begin{note} \label{Note: Normal}
In the case of normal operator $T \in \mathcal{B}(\mathcal{H})$, there exists an anti self-adjoint unitary operator $J \in \mathcal{B}(\mathcal{H})$ commutes with $T$  by Theorem \ref{Theorem: Cartesian}. Then by Proposition \ref{Proposition: Key} there is a complex linear operator, we denote it by  $T_{+} \in \mathcal{B}(\mathcal{H}^{Ji}_{+})$  such that $T = \widetilde{T}_{+}$.	
\end{note}
\section{Factorization in a strongly irreducible sense}
One of the fundamental result in the direction of factorizing  quaternionic operators is the well known polar decomposition theorem \cite{Ghiloni, GandP}. It states that if $T$ is bounded or densely defined closed operator on a right quaternionic Hilbert space $\mathcal{H}$, then there exists a unique partial isometry $W_{0} \in \mathcal{B}(\mathcal{H}$) satisfying 
\begin{equation}\label{Equation: partial}
T = W_{0}|T| \;\; \text{and }\; N(T)=N(W_{0}).
\end{equation}

Recently, the authors of \cite{GandP} obtained a necessary and sufficient condition for any arbitrary decomposition to coincide with the polar decomposition given in Equation (\ref{Equation: partial}). We recall the result here.
\begin{thm}\cite[Theorem 5.13]{GandP}
	Let $T$ be a bounded or densely defined closed operator defined on a right quaternionic Hilbert space $\mathcal{H}$. If  $W \in \mathcal{B}(\mathcal{H})$ is a partial isometry satisfying $T = W|T|$, then $W = W_{0}$ if and only if either $N(T) = \{0\}$ or $R(T)^{\bot} = \{0\}$, where $W_{0}$ is the partial isometry satisfying Equation (\ref{Equation: partial}).
\end{thm}

In this section, firstly, we adopt the notion of strong irreducibility \cite{Gilfeather} to the class of bounded quaternionic operators, and prove a relation between strong irreducibility and the spherical point spectrum. Later, we prove a factorization of quaternionic normal operator in a  strongly irreducible sense, by means of replacing the partial isometry $W$ by a desirably small compact perturbation of $W$ and $|T|$ is replaced by strongly irreducible operator. It is a quaternionic extension (for normal operators) of the result proved in \cite{Tian}.
\begin{defn}\label{Definition: stronglyirreducible}\cite{Gilfeather}
	Let $T \in \mathcal{B}(\mathcal{H})$. Then $T$ is said to be 
	\begin{enumerate}
		\item {\it irreducible}, if there {\bf does not} exist a nontrivial orthogonal projection $P \in \mathcal{B}(\mathcal{H})$ (i.e., $P \neq 0$ and $P\neq I$)   such that $PT = TP$. Otherwise, $T$ is called {\it reducible}.
		\item {\it strongly irreducible}, if there {\bf  does not}  exist a nontrivial idempotent $E \in \mathcal{B}(\mathcal{H})$ (i.e., $E \neq 0$ and $E \neq I$) such that $TE = ET$. Otherwise, $T$ is called {\it strongly reducible}.
	\end{enumerate}
	It is clear from Definition \ref{Definition: stronglyirreducible} that every strongly irreducible operator is irreducible. Similar to the classical setup, the class of strongly irreducible quaternionic operators is closed under similarity invariance.
\end{defn}
In order to describe strong irreducibility or irreducibility of quaternionic normal operators, we show that it is enough to deal with the corresponding complex linear operator defined on slice Hilbert space. 
\begin{lemma}\label{Lemma: Key}
	Let $J \in \mathcal{B}(\mathcal{H})$ be anti self-adjoint unitary operator and $S \in \mathcal{B}(\mathcal{H}^{Ji}_{+})$, then 
	\begin{enumerate}
		\item $\widetilde{S}$ is irreducible if and only if ${S}$ is irreducible.
		
		\item  $\widetilde{S}$ is strongly irreducible if and only if ${S}$ is strongly irreducible. 
	\end{enumerate}  
\end{lemma}
\begin{proof}
	Proof of $(1):$ Suppose that $\widetilde{S}$ is irreducible, then we show that $S$ is irreducible. If there is an orthogonal projection $ 0 \neq P\in \mathcal{B}(H^{Ji}_{+})$ such that $SP = PS$, then  by Proposition \ref{Proposition: Key}, we see that
	\begin{equation*}
	(\widetilde{P})^{\ast} = \widetilde{P^{\ast}} = \widetilde{P}\;\; \text{and}\;\; (\widetilde{P})^{2} = \widetilde{P^{2}} = \widetilde{P}.
	\end{equation*}
	Further,
	\begin{equation*}
	\widetilde{S}\widetilde{P} = \widetilde{SP} = \widetilde{PS} = \widetilde{P} \widetilde{S}. 
	\end{equation*}
	This shows that $0\neq \widetilde{P} \in \mathcal{B}(\mathcal{H})$ is an othogonal projection which commutes with  $\widetilde{S}$. Since  $\widetilde{S}$ is irreducible, we conclude that $\widetilde{P} = I$, the identity operator on $\mathcal{H}$. Thus $P$ is an identity operator on $\mathcal{H}^{Ji}_{+}$ and hence $S$ is irreducible. Now we prove contrapositive statement. Suppose that $\{0\} \neq \mathcal{M} \subsetneqq \mathcal{H}$ is a reducing subspace for $\widetilde{S}$. Define 	
	\begin{equation*}
	\mathcal{M}^{Ji}_{+}: = \{x \in \mathcal{M} :\; Jx = x\cdot i\} = \mathcal{M} \cap \mathcal{H}^{Ji}_{+}.
	\end{equation*}
	For $x \in \mathcal{M}^{Ji}_{+}$, we have 
	\begin{equation*}
	JS(x) = J\widetilde{S}(x) 	= \widetilde{S}J(x)	= \widetilde{S}(x\cdot i)	= \widetilde{S}(x)\cdot i	= S(x)\cdot i
	\end{equation*}
	and 
	\begin{equation*}
	JS^{\ast}(x) = J\widetilde{S^{\ast}}(x) = \widetilde{S^{\ast}}J(x) = \widetilde{S^{\ast}}(x\cdot i) = \widetilde{S^{\ast}}(x)\cdot i= S^{\ast}(x)\cdot i.
	\end{equation*}
	This implies that $\mathcal{M}^{Ji}_{+}$ is a reducing subspace of $S$. It is enough to show that $\mathcal{M}^{Ji}_{+}$ is a non-trivial proper subspace of $\mathcal{H}^{Ji}_{+}$.
	
	\noindent {\bf Claim}: $\{0\} \neq \mathcal{M}^{Ji}_{+} \subsetneqq  \mathcal{H}^{Ji}_{+}$. 
	
	Firstly, we assume that  $ \mathcal{M}^{Ji}_{+} = \{0\} $. In this case, $\mathcal{M}^{Ji}_{-} = \{0\}$ since  $x \in \mathcal{M}^{Ji}_{-}$ if and only if $x\cdot j \in \mathcal{M}^{Ji}_{+}$, where $i\cdot j = - j \cdot i$. Therefore, $M = M^{Ji}_{+} \oplus M^{Ji}_{-}  = \{0\}$. This is a contradiction to the fact that $\mathcal{M}$ is a non trivial subspace of $\mathcal{H}$. Hence $\mathcal{M}^{Ji}_{+} \neq \{0\}$. 
	
	Secondly, assume that $\mathcal{M}^{Ji}_{+} = \mathcal{H}^{Ji}_{+}$. Let $y\in \mathcal{H}$ with $y = y_{+} + y_{-}$, where $y_{\pm} \in \mathcal{H}^{Ji}_{\pm}$. We know that $y_{-} \cdot j \in \mathcal{H}^{Ji}_{+} = \mathcal{M}^{Ji}_{+} $ and so $y_{-} \in \mathcal{M}$. Therefore $y \in \mathcal{M}$. This is contradiction to the fact that $\mathcal{M}$ is a proper subspace of  $\mathcal{H}$. Hence $\mathcal{M}^{Ji}_{+} \subsetneqq \mathcal{H}^{Ji}_{+}$.  We conclude that $\widetilde{S}$ is irreducible if and only if $S$ is irreducible.

	%
	%
	Proof of $(2):$ Suppose that $\widetilde{S}$ is strongly irreducible.  If there is an idempotent $0 \neq E \in \mathcal{B}(\mathcal{H}^{Ji}_{+})$  such that $SE = ES$, then by Proposition \ref{Proposition: Key} we have that $0 \neq \widetilde{E} \in \mathcal{B}(\mathcal{H})$ with $\widetilde{E}^{2} = \widetilde{E}$ and
	\begin{equation*}
	\widetilde{S}\widetilde{E} = \widetilde{SE} = \widetilde{ES} = \widetilde{E} \widetilde{S}.
	\end{equation*}
	Since $\widetilde{S}$ is strongly irreducible, we conclude that $\widetilde{E} = I$, the identity operator on $\mathcal{H}$. This implies that $E$ is the identity operator on $\mathcal{H}^{Ji}_{+}$ and hence  $S$ is strongly irreducible. 
	
	Conversely, assume that $S$ is strongly irreducible. If there is an idempotent say $0 \neq F \neq I $ in $\mathcal{B}(\mathcal{H})$ such that $\widetilde{S}F = F\widetilde{S}$. Since $R(F)$ is a closed subspace of $\mathcal{H}$, we have  $\mathcal{H}^{Ji}_{+} = {R(F)}^{Ji}_{+} \oplus (R(F)^{Ji}_{-})^{\bot}$, where 
	\begin{equation*}
	R(F)^{Ji}_{+} = \{x\in R(F):\; J(x) = x\cdot i\} = R(F)\cap \mathcal{H}^{Ji}_{+}.
	\end{equation*}
	Furthermore, if $x \in R(F)^{Ji}_{+}$, then
	\begin{align*}
	JS(x) = J\widetilde{S}(x) = \widetilde{S}J(x) = \widetilde{S}(x\cdot i) = \widetilde{S}(x)\cdot i = S(x)\cdot i. 
	\end{align*}
	and
	\begin{equation*}
	JS^{\ast}(x) = J\widetilde{S^{\ast}}(x) = \widetilde{S^{\ast}}J(x)=\widetilde{S^{\ast}}(x\cdot i)= \widetilde{S^{\ast}}(x)\cdot i= S^{\ast}(x)\cdot i. 
	\end{equation*}
	This shows that $S$ is reducible. It is a contradiction to the fact that $S$ is irreducible. Hence $\widetilde{S}$ is strongly irreducible.  \qedhere
\end{proof}
Now we prove the relation between spherical point spectrum and strong irreduicibility. It is a quaternionic analogue of the result proved by F. Gilfeather in \cite{Gilfeather}.
\begin{thm}\label{Theorem: emptypointspectrum}
	Let $T \in \mathcal{B}(\mathcal{H})$ be normal. If the spherical point spectrum of $T$ is empty set \big(i.e., 
	$\sigma_{p^{S}}(T) = \emptyset$\big), then $T$ is strongly irreducible.
\end{thm}
\begin{proof}
	Since $T$ is normal, then by Note \ref{Note: Normal} there is an anti self-adjoint unitary operator $J \in \mathcal{B}(\mathcal{H})$  commuting with $T$ such that $ T = \widetilde{T_{+}}$, where $T_{+} \in \mathcal{B}(\mathcal{H}^{Ji}_{+})$ is normal operator. Now we show that the point spectrum $\sigma_{p}(T_{+})$ of $T_{+}$ is empty. Suppose that $\lambda \in \sigma_{p}(T_{+})$, then 
	\begin{equation*}
	T_{+}(x) = x \cdot \lambda, \; \text{for some}\; x \in \mathcal{H}^{Ji}_{+}\setminus \{0\}.
	\end{equation*}  
	This implies that 
	\begin{align*}
	\Delta_{\lambda}(T)x 
	&= \big( T^{2}-2 re(\lambda)T+|\lambda|^{2}I\big)x\\
	&= x \lambda^{2} - 2 x \lambda re(\lambda) + x |\lambda|^{2}\\
	&= x \big( \lambda^{2}- \lambda (\lambda + \overline{\lambda}) + |\lambda|^{2}\big)\\
	&=0.
	\end{align*}
	Equivalently, $x \in N(\Delta_{\lambda}(T)) \neq \{0\}$. This shows that 
	\begin{equation*}
	    \lambda \in \sigma_{p}(T_{+}) \; \Leftrightarrow \; [\lambda] \in \sigma_{p^{S}}(T).
	\end{equation*}
	If follows that $\sigma(T_{+}) = \emptyset \; \big($since $\sigma_{p^{S}}(T) = \emptyset$\big). Since $T_{+}$ is a bounded complex normal operator with empty point spectrum, then by \cite[Theorem 2]{Gilfeather} the operator $T_{+}$ is strongly irreducible. Finally, by Lemma \ref{Lemma: Key} we conclude that  $T$ is strongly irreducible.
\end{proof}
Now we prove that every quaternionic normal operator can be factorized in a strongly irreducible sense.

\begin{thm}
	Let $T \in \mathcal{B}(\mathcal{H})$ be normal and $\delta >0$. Then there exist a partial isometry $W$, a compact operator $K$ with $\|K\| < \delta $ and a strongly irreducible operator $S$ in $\mathcal{B}(\mathcal{H})$ such that 
	\begin{equation*}
	T = (W+K)S.
	\end{equation*}
\end{thm}
\begin{proof}
	Since $T$ is normal, by Theorem \ref{Theorem: Cartesian} and Note \ref{Note: Normal},  there exists an anti self-adjoint unitary operator $J\in \mathcal{B}(\mathcal{H})$ commuting with $T$ such that $T = \widetilde{T_{+}}$, where $T_{+} \in \mathcal{B}(\mathcal{H}^{Ji}_{+})$ is a normal operator. It is clear from \cite[Theorem 1.1]{Tian} that for a given $\delta >0$, there exists a partial isometry $W_{+}$, a compact operator $K_{+}$ with $\|K_{+}\| < \delta $ and a strongly irreducible operator $S_{+}$ in $\mathcal{B}(\mathcal{H}^{Ji}_{+})$ such that
	\begin{equation}\label{Equation: forT+}
	T_{+} = (W_{+}+K_{+}) S_{+}.
	\end{equation} 
	If we define $W := \widetilde{W}_{+}$, then we see that $W$ is a partial isometry as follows: for every $x \in N(W)^{\bot}$, there exists $x_{\pm} \in N(W)^{\bot}\cap \mathcal{H}^{Ji}_{\pm}$ such that $x = x_{+}+x_{-}$ and  
	\begin{align*}
	\|Wx\|^{2} &= \|W_{+}(x_{+}) - W_{+}(x_{-}\cdot j) j\|^{2}\\
	&=  \|W_{+}(x_{+})\|^{2} + \|W_{+}(x_{-}\cdot j) \|^{2}\\
	&= \|x_{+}\|^{2}+ \|x_{-}\|^{2}\\
	&= \|x\|^{2}.
	\end{align*}
	Now we show that the operator defined by $K:= \widetilde{K}_{+}$ is a quaternionic compact operator on $\mathcal{H}$. Since $K_{+}$ is a compact operator on $\mathcal{H}^{Ji}_{+}$, there is a sequence of finite rank operators $\{F_{n}:\; n\in \mathbb{N}\} \subset \mathcal{B}(\mathcal{H}^{Ji}_{+})$  converging to $K_{+}$ (uniformly) with respect to the topology induced from the operator norm. Then by $(1)$ of Proposition \ref{Proposition: Key} we see that 
	\begin{equation*}
	\|\widetilde{F}_{n}-K\| = \|\widetilde{F}_{n}-\widetilde{K}_{+}\| = \|F_{n}-K_{+}\| \longrightarrow 0, \; \text{as}\; n \to \infty.
	\end{equation*}
	This implies that the sequence $\{\widetilde{F}_{n}:\ n \in \mathbb{N}\}\subset \mathcal{B}(\mathcal{H})$ of finite rank quaternionic operators  converges to $K$ uniformly. Thus the operator $K$ is compact and its norm is given by
	\begin{equation*}
	    \|K\| = \|K_{+}\| < \delta.
	\end{equation*}
	Moreover, by Lemma \ref{Lemma: Key}, the quaternionic operator defined by $S:= \widetilde{S}_{+}\in \mathcal{B}(\mathcal{H})$ is strongly irreducible. Now we apply quaternionic extension to bounded complex linear operator $T_{+}$ and use its factorization given in  Equation (\ref{Equation: forT+}), we conclude that
	\begin{equation*}
	T = \widetilde{T}_{+} = (\widetilde{W}_{+}+ \widetilde{K}_{+})\widetilde{S}_{+}= (W+K)S. 
	\end{equation*}
	Hence the result. \qedhere
\end{proof}
We illustrate our result with the following example.
\begin{eg}\label{Example: normal}
	Let $T \colon L^{2}\big([0,1]; \mathbb{H}; \mu\big) \to L^{2}\big([0,1]; \mathbb{H}; \mu\big)$ be defined by 
	\[ (Tg)(x) = \left\{\begin{array}{cc}
	x g(x) + \frac{1}{2} \int\limits_{0}^{1} xy^{2} g(y)\; dy, & \mbox{ if \; $0 \leq x \leq \frac{1}{3}$}; \\
	\frac{1}{2} \int\limits_{0}^{1} xy^{2} g(y)\; dy, & \mbox{ if \; $\frac{1}{3} \leq x \leq 1$},
	\end{array}
	\right.
	\] 
	for all $g \in L^{2}\big([0,1]; \mathbb{H}; \mu\big)$. Then the adjoint of $T$ is given by
	\[ (T^{\ast}g)(x) = \left\{\begin{array}{cc}
	x g(x) + \frac{1}{2} \int\limits_{0}^{1} x^{2}y g(y)\; dy, & \mbox{ if \; $0 \leq x \leq \frac{1}{3}$}; \\
	\frac{1}{2} \int\limits_{0}^{1} x^{2}y g(y)\; dy, & \mbox{ if \; $\frac{1}{3} \leq x \leq 1$},
	\end{array}
	\right.
	\] 
	for all $g \in L^{2}\big([0,1]; \mathbb{H}; \mu\big)$.  Clearly, $T$ is normal. Suppose that $\delta = \frac{1}{2}$. Now we factorize $T$ in a strongly irreducible sense. Define the integral operator  $K \colon L^{2}\big([0,1]; \mathbb{H}; \mu \big) \to L^{2}\big([0,1]; \mathbb{H}; \mu\big) $ by
	\begin{equation*}
	(Kg)(x) = \frac{1}{2} \int\limits_{0}^{1}xy\;g(y)dy, \; \text{for all}\; g \in L^{2}\big([0,1]; \mathbb{H}; \mu \big).
	\end{equation*}
	It is well known that $K$ is a compact operator. By using Caucy-Schwarz inequality, the norm of $K$ is computed as, 
	\begin{align*}
	\|Kg\|_{2} =   \Big(\int\limits_{0}^{1} |(Kg)(x)|^{2}\; dx\Big)^{\frac{1}{2}} 
	&\leq \Big(\int\limits_{0}^{1} \int\limits_{0}^{1} |xy|^{2}\; |g(y)|^{2}dy dx \Big)^{\frac{1}{2}}\\
	&\leq \Big( \int\limits_{0}^{1} |g(y)|^{2}\; dy\Big)^{\frac{1}{2}} \Big(\int\limits_{0}^{1} \int\limits_{0}^{1} |xy|^{2}dy dx \Big)^{\frac{1}{2}}\\
	&= \|g\|_{2} \Big(\int\limits_{0}^{1} \int\limits_{0}^{1} x^{2}y^{2}dy dx \Big)^{\frac{1}{2}}\\
	&= \frac{1}{3} \|g\|_{2}. 
	\end{align*}
	This shows that $\|K\| \leq \frac{1}{3} < \frac{1}{2}$. We recall that the class of all bounded $\mathbb{H}$-valued measurabale functions on $[0,1]$ is denoted by $L^{\infty}\big([0,1]; \mathbb{H}; \mu\big)$. For every $f \in  L^{\infty}\big([0,1]; \mathbb{H}; \mu\big)$, the multiplication operator $M_{f}:L^{2}\big([0,1]; \mathbb{H}; \mu\big) \to L^{2}\big([0,1]; \mathbb{H}; \mu\big)$ defined by
	\begin{equation*}
	M_{f}(g)(x) = f(x)g(x), \; \text{for all}\; g \in L^{2}\big([0,1]; \mathbb{H}; \mu\big)
	\end{equation*}
	is a bounded quaternionic operator with the norm $\|M_{f}\| = \|f\|_{\infty}$. The adjoint of $M_{f}$ is given by $M_{f}^{\ast}= M_{\overline{f}}$, where $\overline{f}(x) = \overline{f(x)}$ for all $x \in [0,1]$. Let $\phi(x) = x$, for all $x \in [0,1]$ and the characteristic function
	\[ \rchi_{[0, \frac{1}{3}]} = \left\{ \begin{array}{cc}
	1, & \mbox{if $x \in [0, \frac{1}{3}]$};  \\
	0, & \mbox{otherwise}.
	\end{array} \right.\]
	Then by the direct verification, we get that
	\begin{equation}\label{Equation: desiredfactorization}
	T = \Big(M_{\rchi_{[0, \frac{1}{3}]}} + K\Big) M_{\phi}. 
	\end{equation}
	Note that $M_{\rchi_{[0,\frac{1}{3}]}}$ is a partial isometry, and since $M_{\phi}$ is normal with $\sigma_{p^{S}}(M_{\phi}) = \emptyset$, then  $M_{\phi}$ is strongly irreducible by Theorem \ref{Theorem: emptypointspectrum}. Therefore, the factorization of $T$ given in Equation (\ref{Equation: desiredfactorization}) is a strongly irreducible factorization. 
\end{eg}
Now we contruct an example of a non-normal operator by a slight modification of the linear operator defined in Example \ref{Example: normal}  and compute its strongly irreducible factorization.
\begin{eg}\label{Example: non-normal}
	Let us define $T \colon L^{2}\big([0,1]; \mathbb{H}; \mu\big) \to L^{2}\big([0,1]; \mathbb{H}; \mu\big)$ by 
	\[ T(g)(x) = \left\{ \begin{array}{cc}
	x g(x) + \frac{j}{2}\int\limits_{0}^{x} y g(y) dy, & \mbox{if\; $0 \leq x \leq \frac{1}{3}$};\\
	\frac{j}{2}\int\limits_{0}^{x}yg(y) dy, & \mbox{if\; $\frac{1}{3} < x \leq 1$},
	\end{array} \right. \]
	for all $g \in L^{2}\big([0,1]; \mathbb{H}; \mu\big)$ and suppose that $\delta = \frac{1}{2}$. Firstly, we show that $T$ is a bounded quaternionic non-normal operator. Let $g,h \in L^{2}\big([0,1]; \mathbb{H}; \mu\big)$. Then 
	\begin{align*}
	\big\langle h,&\; Tg\big\rangle \\
	&= \int\limits_{0}^{1} \overline{h(x)}\; {(Tg)(x)}\; dx \\
	&= \int\limits_{0}^{\frac{1}{3}} \overline{h(x)}\; \Big[x{g(x)} + \frac{j}{2}\int\limits_{0}^{x} y{g(y)}\;dy \Big] dx + \int\limits_{\frac{1}{3}}^{1} \overline{h(x)}\; \Big[\frac{j}{2}\int\limits_{0}^{x} y{g(y)}\;dy \Big] dx\\
	&= \int\limits_{0}^{\frac{1}{3}} \overline{h(x)}x{g(x)}\; dx + \int\limits_{0}^{\frac{1}{3}} \int\limits_{0}^{x} \overline{h(x)}\; \Big[\frac{j}{2} y g(y)\Big]\; dy dx + \int\limits_{\frac{1}{3}}^{1} \int\limits_{0}^{x} \overline{h(x)}\; \Big[\frac{j}{2} y g(y)\Big]\; dy dx.
	\end{align*}
	
	By Fubini's theorem, the above integral can be written as,
	\begin{align*}
	\big\langle h,&\; Tg\big\rangle \\
	&= \int\limits_{0}^{\frac{1}{3}} \overline{h(y)}\;y{g(y)}\; dy + \int\limits_{0}^{\frac{1}{3}} \int\limits_{y}^{\frac{1}{3}} \overline{h(x)}\; \Big[\frac{j}{2} y g(y)\Big]\; dx dy + \int\limits_{\frac{1}{3}}^{1} \int\limits_{y-\frac{1}{3}}^{y} \overline{h(x)}\; \Big[\frac{j}{2} y g(y)\Big]\; dx dy\\
	&= \int\limits_{0}^{\frac{1}{3}} \overline{{yh(y)}}\;{g(y)}\; dy + \int\limits_{0}^{\frac{1}{3}} \Big[ \overline{\frac{-j}{2}y \int\limits_{y}^{\frac{1}{3}} { h(x)}\; dx} \Big]\; g(y)\; dy+\int\limits_{\frac{1}{3}}^{1} \Big[ \overline{\frac{-j}{2}y \int\limits_{y-\frac{1}{3}}^{y} { h(x)}\; dx} \Big]\; g(y)\; dy.
	\end{align*}
	Thus the adjoint of $T$ is give by
	\[ (T^{\ast}h)(y) = \left\{ \begin{array}{cc}
	y g(y) -\frac{j}{2} y  \int\limits_{y}^{\frac{1}{3}}\; h(x)\; dx, & \; \mbox{if \; $0 \leq y \leq \frac{1}{3}$}; \\
	-\frac{j}{2} y  \int\limits_{y-\frac{1}{3}}^{y}\; h(x)\; dx, & \mbox{if \; $\frac{1}{3} \leq y \leq 1$}.
	\end{array}\right.\]
	It follows that $TT^{\ast} \neq T^{\ast}T$. Now we show that $T$ can be factorized in a strongly irreducible sense. Firstly, we define $K \colon L^{2}\big([0,1]; \mathbb{H}; \mu\big) \to L^{2}\big([0,1]; \mathbb{H}; \mu\big)$ by 
	\begin{equation*}
	(Kg)(x) = \frac{j}{2}\int\limits_{0}^{x}g(t)\; dt, \; \text{for all}\; g \in L^{2}\big([0,1]; \mathbb{H}; \mu\big).
	\end{equation*}
	Our aim to show that $K$ is a compact operator with $\|K\| < \frac{1}{2}$.  Let $\{g_{n}\}_{n \in \mathbb{N}}$ be a sequence in $ L^{2}\big([0,1]; \mathbb{H}; \mu\big)$ with $\|g_{n}\| \leq 1$, for all $n \in \mathbb{N}$. Then  
	\begin{equation}\label{Equation: uniformlybounded}
	|(Kg_{n})(x)|  = \Big| \frac{j}{2} \int\limits_{0}^{x}g_{n}(t)\; dt \Big| \leq \frac{1}{2} \int\limits_{0}^{x} |g_{n}(x)| \; dt \leq \frac{1}{2},
	\end{equation}
	for all $x \in [0,1]$ and $n \in \mathbb{N}$. Further, by H\"{o}lders inequality, we get
	\begin{equation}\label{Equation: equicontinuous}
	\big|Kg(x) - Kg(y)\big| = |\frac{j}{2}| \Big| \int\limits_{0}^{x}g(t)\;dt - \int\limits_{0}^{y}g(t)\; dt\Big| \leq \frac{1}{2} \int\limits_{y}^{x}\big|g(t)\big|\; dt \leq \frac{1}{2} \|g\|_{2} \sqrt{|x-y|}.
	\end{equation}
	It follows from Equations (\ref{Equation: uniformlybounded}), (\ref{Equation: equicontinuous}) that the sequece $\{Kg_{n}\}_{n\in \mathbb{N}}$ is uniformly bounded and equicontinuous. By Arzela-Ascoli's theorem, there is a subseqeuce $\{g_{n_{k}}\}$ of $\{g_{n}\}_{n\in \mathbb{N}}$ such that $\{Kg_{n_{k}}\}$ converges uniformly. Thus $K$ is a compact operator.
	
	Now we compute the norm of $K$. Firstly, by applying the Fubini's theorem, we get the adjoint of $K$ as, 
	\begin{equation*}
	(K^{\ast}g)(x) = \frac{-j}{2} \int\limits_{x}^{1} g(t)\; dt, \; \text{for all} \; g \in L^{2}\big([0,1]; \mathbb{H}; \mu\big). 
	\end{equation*}
	So the operator $K^{\ast}K$ is given by,
	\begin{equation*}
	(K^{\ast}Kg)(x) = \frac{-j}{2} \int\limits_{x}^{1}\Big( \frac{j}{2} \int\limits_{0}^{t}g(s)\; ds\Big) dt = \frac{1}{4} \int\limits_{x}^{1}\int\limits_{0}^{t}g(s)\; ds dt
	\end{equation*}
	is a positive quaternionic compact operator. We know from \cite[Corollary 2.13]{GandP-MultiplicationForm} that $L^{2}\big([0,1]; \mathbb{C}; \mu\big)$ is an associated slice Hilbert space and let $(K^{\ast}K)_{+}$ be the bounded complex linear operator on $L^{2}\big([0,1]; \mathbb{C}; \mu\big)$ such that $\widetilde{(K^{\ast}K)}_{+} = K^{\ast}K$. Then by norm of the Voterra integral operator computed as in \cite[Solution 188]{Halmos} and $(1)$ of Proposition \ref{Proposition: Key}, we conclude that 
	\begin{equation*}
	\|K\| = \|K^{\ast}K\|^{\frac{1}{2}} = \|(K^{\ast}K)_{+}\|^{\frac{1}{2}} = \Big(\frac{1}{\pi^{2}}\Big)^{\frac{1}{2}} = \frac{1}{\pi} < \frac{1}{2}. 
	\end{equation*}
	Let us take $W:= M_{\chi_{[0, \frac{1}{3}]}}$ and $S = M_{\varphi}$, where $\varphi(x) = x$, for all $x \in [0,1]$. Clearly, $W$ is a partial isometry and since $S$ is normal with $\sigma_{p^{S}}(S) = \emptyset$, we see that $S$ is strongly irreducible from Theorem \ref{Theorem: emptypointspectrum}. Finally, we have that 
	\begin{equation*}
	T = (W+K)S.
	\end{equation*}
\end{eg}
We pose the following question.
\begin{question} \label{Question: Q}
	Let $T \colon \mathcal{D}(T) \subseteq \mathcal{H} \to \mathcal{H}$ be densely defined closed right $\mathbb{H}$- linear operator (need not be normal), where $\mathcal{D}(T)$ is the domain of $T$. Then, can $T$ be  factorized in a strongly irreducible sense ? 
\end{question}
We expect that, by using the notion of quaternionic Cowen-Douglas operators related to geometry of quaternionic Hilbert spaces developed in \cite{Hou} and further suitable arguments, may achieve affirmative answer to the Question \ref{Question: Q}. 

\section{Riesz Decomposition Theorem}
In this section we prove Riesz decomposition theorem for bounded quaternionic operators on right quaternionic Hilbert spaces and obtain a sufficient condition for strong irreducibility. We recall some definitions and known results form \cite{Alpay-R, Colombo, Ghiloni} that are useful to establish our result.
\begin{defn}\cite{Alpay-R}
	Let  $U \subseteq \mathbb{H}$ be an open set. Then 
	\begin{enumerate}
	    \item $U$ is said to be a \emph{slice domain} or \emph{$s$-domain}, if $U$ is a domain in $\mathbb{H}$ such that $U \cap \mathbb{R} \neq \emptyset$ and $U\cap \mathbb{C}_{m}$ is domain in $\mathbb{C}_{m}$, for all $m \in \mathbb{S}$.
	    \item A real differentiable function $f \colon U \to \mathbb{H}$ is said to be 
	            \begin{enumerate}
	         	\item[(i)] {\it left s-regular}, if for every $m \in \mathbb{S}$, the function $f$ satisfies
	         	\begin{equation*}
	         	\frac{1}{2} \Big[ \frac{\partial{}}{\partial{x}} f(x + m y) + m \; \frac{\partial{}}{\partial{y}} f(x + m y) \Big] = 0
	        	\end{equation*}
		     on $U\cap \mathbb{C}_{m}$. 
		        \item[(ii)] {\it right s-regular}, if for every $m \in \mathbb{S}$, the function $f$ satisfies
		        \begin{equation*}
		        \frac{1}{2} \Big[ \frac{\partial{}}{\partial{x}} f(x + m y) +  \frac{\partial{}}{\partial{y}} f(x + m y)\; m \Big] = 0
		        \end{equation*}
		        on $U\cap \mathbb{C}_{m}$.  
            	\end{enumerate}
	\end{enumerate}
Note that the class of left and right $s$-regular functions defined on $U$ is denoted by $\mathcal{R}^{L}(U)$ and $\mathcal{R}^{R}(U)$, respectively. One can verify that $\mathcal{R}^{L}(U)$ is a right $\mathbb{H}$-module, whereas $\mathcal{R}^{R}(U)$ is a left $\mathbb{H}$-module. 
\end{defn}
The following theorem describes a quaternionic analog of the \emph{Cauchy integral formula} for $s$-regular functions. 
\begin{thm}\cite[Theorem 4.5.3]{Colombo} \label{Theorem: CauchyIntegralFormula}
	Let $U \subseteq \mathbb{H}$ be an axially symmetric $s$-domain such that the boundary $\partial{(U\cap \mathbb{C}_{m})}$ is the union of a finite number of continuously differentiable Jordan curves for every $m \in \mathbb{S}$. Let $W$ be an open set containing $\overline{U}$ and take $ds_{m} = -ds \cdot m$ for any $m \in \mathbb{S}$. We have the following: 
	\begin{enumerate}
		\item If $f \colon W \to \mathbb{H}$ is left $s$-regular function, then 
		\begin{equation}\label{Equation:LeftCauchy}
		f(q) = -\frac{1}{2\pi} \int\limits_{\partial{(U\cap \mathbb{C}_{m})}} \big(q^{2}- 2 re(s)q+|s|^{2}\big)^{-1}(q-\overline{s}) ds_{m} f(s),
		\end{equation}
		for all $q\in U$. 
		\item If $f \colon W \to \mathbb{H}$ is right $s$-regular function, then 
		\begin{equation}\label{Equation:ReftCauchy}
		f(q) = -\frac{1}{2\pi} \int\limits_{\partial{(U\cap \mathbb{C}_{m})}} f(s) ds_{m} (q-\overline{s}) \big(q^{2}- 2 re(s)q+|s|^{2}\big)^{-1},
		\end{equation}
		for all $q\in U$. 
	\end{enumerate}
	Moreover, the  integrals appear in Equations (\ref{Equation:LeftCauchy}), (\ref{Equation:ReftCauchy}) does not depend on the choice of the imaginary unit $m \in \mathbb{S}$ and on $U$.
\end{thm}
Note that  the kernels $- \big(q^{2}- 2 re(s)q+|s|^{2}\big)^{-1}(q-\overline{s}) $ and $- (q-\overline{s}) \big(q^{2}- 2 re(s)q+|s|^{2}\big)^{-1}$ in the Theorem \ref{Theorem: CauchyIntegralFormula} as the limit of corresponding Cauchy kernel series $\sum\limits_{n=0}^{\infty}q^{n} s^{-1-n}$ and $\sum\limits_{n=0}^{\infty}s^{-1-n}q^{n}$, respectively for $|q|< |s|$.
\subsection*{The quaternionic functional calculus}
Let $\mathcal{H}$ be a right quaternionic Hilbert space and let $\mathcal{N}$ be a Hilbert basis of $\mathcal{H}$. It is immediate to see that the class of all bounded right $\mathbb{H}$- linear operators denoted by $\mathcal{B}(\mathcal{H})$ is a two sided quaternionic Banach module with respect to the module actions given by
\begin{equation*}
(q \cdot T)(x) := \sum\limits_{z \in \mathcal{N}} z \cdot q \ \langle z, Tx\rangle \;\; \text{and}\;\; (T\cdot q)(x) := \sum\limits_{z \in \mathcal{N}} T(z)\cdot q\ \langle z, x\rangle, 
\end{equation*}  
for all $T \in \mathcal{B}(\mathcal{H}), q\in \mathbb{H}, x \in \mathcal{H}$. In particular, for an identity operator $I \in \mathcal{B}(\mathcal{H})$, we have
\begin{equation*}
    (q \cdot I)(x) = \sum\limits_{z \in \mathcal{N}} z \cdot q \ \langle z, x\rangle = (I\cdot q)(x), \;\text{for all}\; x \in \mathcal{H}, q \in \mathbb{H}.
\end{equation*}
Next, we recall the notion of the spherical resolvent operator and  the spherical resolvent equation which plays a vital role in establishing quaternionic functional calculus.
\begin{defn}\cite[Definition 4.8.3]{Colombo}\label{Definition: sliceregular}
	Let $T\in \mathcal{B}(\mathcal{H})$ and $s \in \rrho_{S}(T)$. Then the \emph{left spherical resolvent operator} is defined by 
	\begin{equation*}
	S_{L}^{-1}(s, T):= - {\Delta_{s}(T)}^{-1} (T-\overline{s}I) = \sum\limits_{n=0}^{\infty}T^{n}s^{-1-n},
	\end{equation*}
	and the \emph{right spherical resolvent operator} by
	\begin{equation*}
	S_{R}^{-1}(s, T) := - (T-\overline{s}I){\Delta_{s}(T)}^{-1} =\sum\limits_{n=0}^{\infty}s^{-1-n}\; T^{n},
	\end{equation*}
	for $\|T\|<|s|$. 
\end{defn}

\begin{thm}\cite{Colombo} Let $T \in \mathcal{B}(\mathcal{H})$ and $s \in \rrho_{S}(T)$. Then the left and the right Spherical  resolvent operator satisfies the following relations:
	\begin{equation} \label{Equation: S_{L}}
	S_{L}^{-1}(s,T)s- TS_{L}^{-1}(s,T) = I
	\end{equation}
	and
	\begin{equation}\label{Equation: S_{R}}
	s S_{R}^{-1}(s,T)- S_{R}^{-1}(s,T)T=I. 
	\end{equation}
\end{thm}
\begin{proof}
	We know that $T$ commutes with $\Delta_{s}(T)$ and so with $\Delta_{s}(T)^{-1}$. By the Definition \ref{Definition: sliceregular}, we get
	\begin{align*}
	S_{L}^{-1}(s,T)s- TS_{L}^{-1}(s,T) &= -\Delta_{s}(T)^{-1}(T-\overline{s}I)s + T\Delta_{s}(T)^{-1}(T-\overline{s}I)\\
	&= \Delta_{s}(T)^{-1}\Big[-(T-\overline{s}I)s+T(T-\overline{s}I)\Big]\\
	&= \Delta_{s}(T)^{-1}\Delta_{s}(T)\\
	&= I
	\end{align*}
and
	\begin{align*}
	s S_{R}^{-1}(s,T)- S_{R}^{-1}(s,T)T &= s (T-\overline{s}I)\Delta_{s}(T)^{-1} - (T-\overline{s}I)\Delta_{s}(T)^{-1}T \\
	& = \Big[s (T-\overline{s}I)-(T-\overline{s}I)T\Big]\Delta_{s}(T)^{-1}\\
	&= \Delta_{s}(T)\Delta_{s}(T)^{-1}\\
	&= I.
	\end{align*}
	Hence the result.
\end{proof}
\begin{note} Let $A$ be a bounded linear operator on some complex Hilbert space $\mathcal{K}$ and $\lambda \in \rrho(A)$, the resolvent set of $A$. Then $(\lambda I - A)$ is invertible and its inverse is given by the following power series, 
	\begin{equation*}
	(\lambda I - A)^{-1} = \sum\limits_{n=0}^{\infty}\frac{1}{\lambda^{n+1}} A^{n}, \; \; \text{for}\;\; \|A\|< \lambda.
	\end{equation*}
	Moreover, if $\lambda, \mu \in \rho(A)$, then we have the following relation known as resolvent equation:
	\begin{equation}\label{Equation: Requation}
	(\lambda I - A)^{-1} - (\mu I-A)^{-1} = (\mu-\lambda)(\lambda I - A)^{-1}(\mu I - A)^{-1}.
	\end{equation}
%
\end{note} 
One of the crucial observation in estabilshing the quaternionic functional calculus is the spherical resolvent equation whichi is a quaternionic analogue of Equation (\ref{Equation: Requation}). We recall the result here.

\begin{thm}\cite[Theorem 3.8]{Alpay-R}\label{Theorem: S-resolventequation}
	Let $T \in \mathcal{B}(\mathcal{H})$ and let $s,p \in \rrho_{S}(T)$. Then the spherical resolvent equation is given by 
	\begin{align*}
	&S_{R}^{-1}(s, T)S_{L}^{-1}(p, T) \\
	&=  \Big[\big(S_{R}^{-1}(s, T) - S_{L}^{-1}(p, T)\big)p -\overline{s}\big(S_{R}^{-1}(s, T) - S_{L}^{-1}(p, T)\big)\Big] (p^{2}-2\text{re}(s)p+|s|^{2})^{-1}.
	\end{align*}
	Equivalently, 
	\begin{align*}
	&S_{R}^{-1}(s, T)S_{L}^{-1}(p, T) \\
	&=  (s^{2}-2\text{re}(p)s+|p|^{2})^{-1} \Big[\big(S_{R}^{-1}(s, T) - S_{L}^{-1}(p, T)\big)p -\overline{s}\big(S_{R}^{-1}(s, T) - S_{L}^{-1}(p, T)\big)\Big] .
	\end{align*}
\end{thm}

\begin{defn}\label{Definition: locallyregular}
    Let $T \in \mathcal{B}(\mathcal{H})$, $W \subseteq \mathbb{H}$ be open and $U \subseteq \mathbb{H}$ be a domain in $\mathbb{H}$. Then 
    \begin{enumerate}
    \item $U$ is said to be a {\it T -admissible open set}, if $U$ is axially symmetric $s$-domain that contains the spherical spectrum $\sigma_{S}(T)$ such that the boundary $\partial{(U\cap \mathbb{C}_{m})}$ is the union of a finite number of continuously differentiable Jordan curves, for every $m \in \mathbb{S}$.
		\item A function $f \in \mathcal{R}^{L}(W)$ is said to be \emph{ locally left regular} function on $\sigma_{S}(T)$, if there is $T$-admissible domain $U$ in $\mathbb{H}$ such that $\overline{U} \subseteq W$.
		\item A function $f \in \mathcal{R}^{R}(W)$ is said to be  \emph{ locally right regular} function on $\sigma_{S}(T)$, if there is $T$-admissible domain $U$ in $\mathbb{H}$ such that $\overline{U} \subseteq W$.  
	\end{enumerate}
	The class of all locally left and locally right regular functions on $\sigma_{S}(T)$ are denoted by  $\mathcal{R}^{L}_{\sigma_{S}(T)}$ and $\mathcal{R}^{R}_{\sigma_{S}(T)}$ respectively.
\end{defn}

Finally, by using Theorem \ref{Theorem: CauchyIntegralFormula} and the quaternionic version of Hahn Banach theorem \cite[Theorem 4.1.10]{Colombo}, the quaternionic functional calculus is defined as below.
\begin{defn}\cite[Definition 4.10.4]{Colombo}\label{Definition: quaternionicfunctionalcalculus}(quaternionic functional calculus)
	Let $T \in \mathcal{B}(\mathcal{H})$ and  $U\subset \mathbb{H}$ be a $T$-admissible domain. Then 
	\begin{equation}\label{Equation: Leftregular}
	f(T) = \frac{1}{2\pi}\int\limits_{\partial{(U\cap \mathbb{C}_{m}})} S_{L}^{-1}(s, T) ds_{m}\; f(s), \; \text{for all}\; f\in \mathcal{R}^{L}_{\sigma_{S}(T)} 
	\end{equation}
	and 
	\begin{equation}\label{Equation: Rightregular}
	f(T) = \frac{1}{2\pi}\int\limits_{\partial{(U\cap \mathbb{C}_{m}})} f(s) ds_{m} \;S_{R}^{-1}(s, T) , \; \text{for all}\; f\in \mathcal{R}^{R}_{\sigma_{S}(T)} ,
	\end{equation}
	where $ds_{m}= - ds\cdot m$. Note that the integrals that appear in Equations (\ref{Equation: Leftregular}), (\ref{Equation: Rightregular}) are independent of the choice of imaginary unit $m \in \mathbb{S}$ and $T$-admissible domain $U$.

\end{defn}
\subsection*{Riesz decomposition theorem}
Before proving our result, let us discuss the adjoint of the operator $f(T)$ defined as in  Definition \ref{Definition: quaternionicfunctionalcalculus}. 

\begin{rmk}\label{Remark: adjoint}
	Let $T \in \mathcal{B}(\mathcal{H})$ and $W$ be an axially symmetric open set in $\mathbb{H}$. For every $f \colon W \to \mathbb{H}$, we define  $\hat{f} \colon W \to \mathbb{H}$ by 
	\begin{equation*}
	\hat{f}(q) = \overline{f(\overline{q})}, \; \text{for all}\;  q\in W.
	\end{equation*}
 Let $f \in \mathcal{R}^{L}(W)$.  Then for every $m \in \mathbb{S}$ and $x,y \in \mathbb{R}$, we see that
	\begin{align*}
	\frac{\partial{}}{\partial{x}}\hat{f}(x+my) +& m \; \frac{\partial{}}{\partial{y}}\hat{f}(x+my) \\
	&= \frac{\partial{}}{\partial{x}}\overline{f(x-my)} + m \; \frac{\partial{}}{\partial{y}}\overline{f(x-my)}\\
	&= \overline{\frac{\partial{}}{\partial{x}}{f(x-my)} - m \; \frac{\partial{}}{\partial{y}}{f(x-my)}}\\
	&= \overline{\frac{\partial{}}{\partial{x}}{f(x+ \overline{m}y)} + \overline{m} \; \frac{\partial{}}{\partial{y}}{f(x+\overline{m}y)}}\\
	&= 0. 
	\end{align*}
	This show that $\hat{f}\in \mathcal{R}^{L}(W)$. Further, if we assume that $f$ is locally left regular function that is, $f \in  \mathcal{R}^{L}_{\sigma_{S}(T)}$ then by Definition \ref{Definition: locallyregular}, there is a $T$-admissible domain $U$ such that $\overline{U} \subseteq W$. Since $\sigma_{S}(T) = \sigma_{S}(T^{\ast})$ and by the above arguments, we conclude that $\hat{f} \in \mathcal{R}^{L}_{\sigma_{S}(T)}$. Now we compute the adjoint of $f(T)$, whenever $f \in \mathcal{R}^{L}_{\sigma_{S}(T)}$, as follows:
	\begin{align*}
	\big\langle x, f(T)y \big\rangle  &= \frac{1}{2\pi} \int\limits_{\partial{(U\cap \mathbb{C}_{m})}} \Big\langle x,\; S_{L}^{-1}(s, T) y \Big\rangle\; ds_{m}\; f(s)\\
	&= \frac{1}{2\pi} \int\limits_{\partial{(U\cap \mathbb{C}_{m})}} \Big\langle S_{L}^{-1}(s, T)^{\ast} x,\;  y \Big\rangle\; ds_{m}\; f(s) \\
	&= \frac{1}{2\pi} \int\limits_{\partial{(U\cap \mathbb{C}_{m})}} \Big\langle S_{R}^{-1}(\overline{s}, T^{\ast}) x,\;  y \Big\rangle\; ds_{m}\; f(s).
	\end{align*} 
	If we put ${s} = \overline{t}$, then $ds = d\overline{t}$ and $ds_{m} = -d\overline{t}\ m$. Since the integration over the domain $\partial(U\cap \mathbb{C}_{m})$ which is symmetric about the real line, we see that $\overline{d\overline{t}_{m}} = dt_{m}$. Thus above integral can be modified as, 
	\begin{align*}
	\big\langle x, f(T)y \big\rangle  &= \frac{1}{2\pi} \int\limits_{\partial{(U\cap \mathbb{C}_{m})}} \Big\langle S_{R}^{-1}(t, T^{\ast}) x,\;  y \Big\rangle d\overline{t}_{m} f(\overline{t})\\
	&= \frac{1}{2\pi}\overline{ \int\limits_{\partial{(U\cap \mathbb{C}_{m})}} \overline{f(\overline{t})} \; \overline{d\overline{t}_{m}}\Big\langle y, \; S_{R}^{-1}(t, T^{\ast}) x\Big\rangle  }\\
	&= \frac{1}{2\pi} \overline{ \Big\langle y, \; \int\limits_{\partial{(U\cap \mathbb{C}_{m})}} \hat{f}(t) \; dt_{m}\; S_{R}^{-1}(t, T^{\ast}) x\Big\rangle  }\\
	&=\big\langle \hat{f}(T^{\ast})x,\; y \big\rangle,
	\end{align*}
	for all $x,y \in \mathbb{H}$. Therefore, $f(T)^{\ast} = \hat{f}(T^{\ast})$ for all $f \in \mathcal{R}^{L}_{\sigma_{S}(T)}$. Similarly, the result holds true for  $\mathcal{R}^{R}_{\sigma_{S}(T)}$. For further details about algebraic properties of quaternionic functional calculus, we refer the reader to \cite[Proposition 4.11.1]{Colombo}. 
\end{rmk}


In the following lemma, we show that for any compact set in $\mathbb{H}$, there is an axially symmetric $s$-domain such that its intersection with $\mathbb{C}_{m}$ is a Cauchy domain in $\mathbb{C}_{m}$, for every $m \in \mathbb{S}$. 
	\begin{lemma} \label{Lemma: Cauchydomain}
	Let $K$ be an axially symmetric compact subset of $\mathbb{H}$ and $W$ be an axially symmetric $s$- domain containing $K$. Then there is an axially symmetric $s$- domain $U$ with the boundary $\partial{(U\cap \mathbb{C}_{m})}$ is the union of a finite number of continuously differentiable Jordan curves (for every $m \in \mathbb{S}$) such that $K \subset U$ and  $\overline{U} \subset W$. 
\end{lemma}
\begin{proof}
	Let us fix $m \in \mathbb{S}$. We define $K_{m}:= K\cap \mathbb{C}_{m}$ and $W_{m}:= W \cap \mathbb{C}_{m}$. Then  $K_{m}$ is a compact subset of the open set $W_{m}$ and hence $K_{m}$ is separated by a positive distance from the closed set $\mathbb{C}_{m}\setminus W_{m}$. That is, 
	\begin{equation*}
	d(K_{m},\; \mathbb{C}_{m}\setminus W_{m}) := r , \; \text{for some}\; r>0.
	\end{equation*}
	Here $``d"$ is the Eucledian metric on the slice complex plane $\mathbb{C}_{m}$. Being a compact set, if $K_{m}$ is covered by open discs  of radius $\frac{r}{2}$ with center in $K_{m}$, then there is a finite collection $\mathcal{U} = \{U_{1}^{(m)}, U_{2}^{(m)}, \cdots , U_{\ell}^{(m)}\}$ of open discs such that
	\begin{equation*}
	K_{m} \subseteq \bigcup\limits_{t=1}^{\ell} U_{t}^{(m)}.
	\end{equation*}
	It is clear that the set defined by $U_{m}:= \bigcup\limits_{t=1}^{\ell}U_{t}^{(m)}$ is open and the boundary $\partial{U_{m}}$ is a finite union of arcs of  $\partial{U^{(m)}_{t}}$ (circles),  for $t = 1,2,\cdots \ell$. This follows that $\overline{U}_{m} \subseteq W_{m }$. Finally, we define $U:= \Omega_{U_{m}}$. Since the center of each disc $U_{t}^{(m)}$ lies on the real line, we see that $U \cap \mathbb{R} \neq \emptyset$ and $U\cap \mathbb{C}_{m}$ is a domain in $\mathbb{C}_{m}$. That is, $U$ is an axially symmetric $s$-domain containing $K$. Moreover, the boundary  $\partial(U\cap \mathbb{C}_{m})$ is the union of finite number of continuously differentiable Jordan curves by the above argument.  The closure of $U$ follows from Equation (\ref{Equation: closureofOmegaS}) as, 
	\begin{equation*}
	 \overline{U} =\overline{\Omega}_{U_{m}} = \Omega_{\overline{U}_{m}} \subset \Omega_{W_{m}} = W.
	\end{equation*} 
	Hence the result.
\end{proof}
\begin{cor}
	Let $T \in \mathcal{B}(\mathcal{H})$ and let $W \subseteq \mathbb{H}$ be an axially symmetric $s$-domain contining the spherical spectrum $\sigma_{S}(T)$. Then there is a $T$-admissible domain $U$ such that $\overline{U} \subseteq W$.
	\begin{proof}
		Since $\sigma_{S}(T)$ is an axially symmetric compact subset of $\mathbb{H}$ and $W$ is an axially symmetric $s$-domain containing $\sigma_{S}(T)$, then by Lemma \ref{Lemma: Cauchydomain} there exists an axially symmetric $s$-domain $U$ containing the spherical spectrum $\sigma_{S}(T)$ such that $\overline{U} \subseteq W$. Equivalently, $U$ is a  $T$-admissible domain satisfying,  $\overline{U} \subseteq W$.
	\end{proof}
\end{cor}
\begin{thm}\label{Theorem: Riesztheorem}
	Let $T \in \mathcal{B}(\mathcal{H})$ and let $\sigma_{S}(T) = \sigma \cup \tau$, where $\sigma$ and $\tau$ are disjoint nonempty axially symmetric closed subsets of $\sigma_{S}(T)$. Then there exist a pair $\{\mathcal{M}_{\sigma}, \mathcal{M}_{\tau}\}$ of non-trivial invariant subspaces of $T$ such that 
	\begin{equation*}
	\sigma = \sigma_{S}(T|_{\mathcal{M}_{\sigma}}) \; \; \text{and}\; \; \tau = \sigma_{S}(T|_{\mathcal{M}_{\tau}}).
	\end{equation*}
	
\end{thm}
\begin{proof} Let $m \in \mathbb{S}$. From hypothesis, it is clear that  $\sigma \cap \mathbb{C}_{m}$ and $\tau\cap\mathbb{C}_{m}$ are disjoint non-empty compact subsets of the Hausdorff space $\mathbb{C}_{m}$. Then there is a pair of disjoint open sets, say $\mathcal{O}^{(m)}_{\sigma}$ and $\mathcal{O}^{(m)}_{\tau}$ of $\mathbb{C}_{m}$ such that $\sigma\cap \mathbb{C}_{m} \subset \mathcal{O}^{(m)}_{\sigma}$ and $\tau\cap \mathbb{C}_{m}\subset \mathcal{O}^{(m)}_{\tau} $. By axially symmetric propoerty of $\sigma$ and $\tau$, we can write
	\begin{equation*}
	\sigma = \Omega_{\sigma \cap \mathbb{C}_{m}} \subset \Omega_{\mathcal{O}^{(m)}_{\sigma}}\; \text{and }\; \tau = \Omega_{\tau\cap \mathbb{C}_{m}}\subseteq \Omega_{\mathcal{O}^{(m)}_{\tau}}.
	\end{equation*}   
	Note that $\Omega_{\sigma\cap \mathbb{C}_{m}}$ and $\Omega_{\tau \cap \mathbb{C}_{m}}$ are nonempty disjoint $s$-domains in $\mathbb{H}$. By  Lemma \ref{Lemma: Cauchydomain}, there exist a pair of axially symmetric $s$-domains, denote them by $U_{\sigma}$ and $U_{\tau}$, containing compact sets $\sigma$ and $\tau$ respectively. Also, the boundaries $\partial(U_{\sigma}\cap \mathbb{C}_{m})$ and $\partial(U_{\tau} \cap \mathbb{C}_{m})$ are the union of finite number of continuously differentiable Jordan curves satisfying,
	\begin{equation*}
	    \overline{U}_{\sigma} \subseteq \Omega_{\mathcal{O}_{\sigma}^{(m)}} \; \text{and}\; \overline{U}_{\tau} \subseteq \Omega_{\mathcal{O}_{\tau}^{(m)}}. 
	\end{equation*}
	
	Now we define quaternionic operators corresponding to  $\sigma$ and $\tau$ as follows:
	
	\begin{equation}\label{Equation: Psigma}
	    P_{\sigma} = \frac{1}{2\pi} \int\limits_{\partial{(U_{\sigma} \cap \mathbb{C}_{m})}}  ds_{m}\; S_{R}^{-1}(s,T) 
	\end{equation}
{and}
\begin{equation} \label{Equation: Ptau}
    	P_{\tau} = \frac{1}{2\pi} \int\limits_{\partial{(U_{\tau} \cap \mathbb{C}_{m})}}  ds_{m} \; S_{R}^{-1}(s,T),
\end{equation}
	where $ds_{m} = -ds\cdot m$.  
	Now we prove the theorem in four steps. 
	
	\vspace{0.3cm}
	\noindent {\bf Step I:} $P_{\sigma}^{2} = P_{\sigma}$ and $P_{\tau}^{2} = P_{\tau}$.
	\vspace{0.2cm}

	As $U_{\sigma}$ is axially symmetric $s$-domain containing the compact set $\sigma$, then by applying Lemma \ref{Lemma: Cauchydomain}, there exists another axially symmetric $s$-domian $U_{\sigma}^{\prime}$ containing $\sigma$ with the boundary  $\partial{(U_{\sigma}^{\prime}\cap \mathbb{C}_{m})}$ is the union of a finite number of continously differentiable Jordan curves such that $\overline{U_{\sigma}^{\prime}} \subset U_{\sigma}$. If we define $\rchi_{\sigma} \colon\sigma_{s}(T) \to \mathbb{H}$ by
	\[  \rchi_{\sigma}(q) = \left\{ \begin{array}{cc}
	     1, & \mbox{whenever $q \in \sigma$}  \\
	     0, & \mbox{otherwise},
	\end{array}\right.\]
	   then clearly $\rchi_{\sigma} \in \mathcal{R}^{L}_{\sigma_{S}(T)} \cap \mathcal{R}^{L}_{\sigma_{S}(T)}$ and $P_{\sigma} = \rchi_{\sigma}(T)$ by the quaternionic functional calculus.
	Since the integral is independent of the choice of $s$-domain and also using the fact that $\rchi_{\sigma}$ is a locally left $s$-regular function,  we can express it by
\begin{equation}\label{Equation: Psigmaforprime}
\rchi_{\sigma}(T) =  \frac{1}{2\pi}\int\limits_{\partial{(U_{\sigma}^{\prime}\cap \mathbb{C}_{m}})}  S_{L}^{-1}(p,T)\; dp_{m} = P_{\sigma},
\end{equation}
where $dp_{m} = -dp \cdot m$. From Equations (\ref{Equation: Psigma}), (\ref{Equation: Psigmaforprime}) and Theorem \ref{Theorem: S-resolventequation}, we compute $P_{\sigma}^{2}$ as follows, 
\begin{align*}
P_{\sigma}^{2} &= \Big(\frac{1}{2\pi}\int\limits_{\partial{(U_{\sigma}\cap \mathbb{C}_{m})}} ds_{m}\; S_{R}^{-1}(s, T) \Big) \cdot \Big(\frac{1}{2\pi}\int\limits_{\partial{(U_{\sigma}^{\prime}\cap \mathbb{C}_{m})}}  S_{L}^{-1}(p, T)\; dp_{m} \Big) \\
&=\frac{1}{4\pi^{2}}\int\limits_{\partial{(U\cap \mathbb{C}_{m}})} ds_{m} \int\limits_{\partial{(U_{\sigma}^{\prime}\cap \mathbb{C}_{m})}}S_{R}^{-1}(s,T)S_{L}^{-1}(p,T)\;dp_{m}\\
&= \frac{1}{4\pi^{2}}\int\limits_{\partial{(U_{\sigma}\cap \mathbb{C}_{m})}} ds_{m} \int\limits_{\partial{(U_{\sigma}^{\prime}\cap \mathbb{C}_{m})}} \Big( S_{R}^{-1}(s,T)-S_{L}^{-1}(p,T)\Big) p\; (p^{2}-2re(s)p+|s|^{2})^{-1} dp_{m}\\
&\;\;\;\; - \frac{1}{4\pi^{2}}\int\limits_{\partial{(U_{\sigma}\cap \mathbb{C}_{m})}} ds_{m} \int\limits_{\partial{(U_{\sigma}^{\prime}\cap \mathbb{C}_{m})}} \overline{s}\Big( S_{R}^{-1}(s,T)-S_{L}^{-1}(p,T)\Big)(p^{2}-2re(s)p+|s|^{2})^{-1} dp_{m}.
\end{align*}
Now we pause our computation for a while. Let us observe the following arguments. For every $s \in \partial{(U_{\sigma}\cap \mathbb{C}_{m})}$, if we define a map $g_{s} \colon U_{\sigma}^{\prime} \to \mathbb{H}$ by 
\begin{equation*}
    g_{s}(q) = (q^{2} - 2re(s)q + |s|^{2})^{-1},\; \text{for all}\; q \in U_{\sigma}^{\prime}
\end{equation*}
then 
\begin{align*}
\frac{\partial{}}{\partial{x}}&\  g_{s}(x+my) +  \frac{\partial{}}{\partial{y}}\ g_{s}(x+my)m\\
& = -\big[(x+my)^{2} - 2re(s)(x+my)+|s|^{2}\big]^{-2} \big(2x+2my-2re(s)\big)\\
& \hspace{0.6cm} - \big[(x+my)^{2} - 2re(s)(x+my)+|s|^{2}\big]^{-2} \big(2xm-2y-2re(s)m\big)m\\
&= -\big[(x+my)^{2} - 2re(s)(x+my)+|s|^{2}\big]^{-2} \big(2x+2my-2re(s)\big)\\
& \hspace{0.6cm} + \big[(x+my)^{2} - 2re(s)(x+my)+|s|^{2}\big]^{-2} \big(2x+2my-2re(s)\big)\\
&= 0,
\end{align*}
for every $x + my \in U_{\sigma}^{\prime} \cap \mathbb{C}_{m}$. This shows that $g_{s} \in \mathcal{R}^{R}(U_{\sigma}^{\prime})$, for every $s \in \partial{(U_{\sigma}\cap \mathbb{C}_{m})}$. Similarly, if we define $h_{s}(q) = q g_{s}(q)$ for all $q \in  U_{\sigma}^{\prime}$, then $h_{s} \in\mathcal{R}^{R}(U_{\sigma}^{\prime})$, for every $s \in \partial{(U_{\sigma}\cap \mathbb{C}_{m})} $. Thus by residue theorem, we see that  
\begin{equation*}
\int\limits_{\partial{(U^{\prime}_{\sigma}\cap \mathbb{C}_{m})}} g_{s}(p) dp_{m} = 0 \;\; ;\; \; \int\limits_{\partial{(U^{\prime}_{\sigma}\cap \mathbb{C}_{m})}} h_{s}(p) dp_{m} = 0.
\end{equation*}
This implies that 
\begin{equation}\label{Equation: hp}
    \frac{1}{4\pi^{2}} \int\limits_{\partial(U_{\sigma}\cap \mathbb{C}_{m})} ds_{m}  S_{R}^{-1}(s,T)\int\limits_{\partial(U_{\sigma}^{\prime}\cap \mathbb{C}_{m})} h_{s}(p) dp_m = 0
\end{equation}
and
\begin{equation}\label{Equation: gp}
    \frac{1}{4\pi^{2}} \int\limits_{\partial(U_{\sigma}\cap \mathbb{C}_{m})} ds_{m} \overline{s} S_{R}^{-1}(s,T)\int\limits_{\partial(U_{\sigma}^{\prime}\cap \mathbb{C}_{m})} h_{s}(p) dp_m = 0.
\end{equation}
Moreover,  by \cite[Lemma 3.18]{Alpay-R}, we have 
\begin{equation}\label{Equation: Finally}
    \frac{1}{2\pi} \int\limits_{\partial(U_{\sigma}\cap \mathbb{C}_{m})} ds_{m} \big[ \overline{s}S_{L}^{-1}(p,T) - S_{L}^{-1}(p, T)p \big] g_{s}(p)  = S_{L}^{-1}(s,T)
\end{equation}
since $S_{L}^{-1}(s,T) \in \mathcal{B}(\mathcal{H})$. Now we resume our computation of $P_{\sigma}^{2}$. From Equations (\ref{Equation: hp}), (\ref{Equation: gp}) and above arguments, we get that 
\begin{align*}
P_{\sigma}^{2} &= -\frac{1}{4\pi^{2}}\int\limits_{\partial{(U_{\sigma}\cap \mathbb{C}_{m})}}ds_{m} \int\limits_{\partial{(U^{\prime}_{\sigma}\cap \mathbb{C}_{m})}} S_{L}^{-1}(p,T) h_{s}(p)dp_{m}\\
& \;\;\;\;\;\;\;\;\;\;+\frac{1}{4\pi^{2}}\int\limits_{\partial{(U_{\sigma}\cap \mathbb{C}_{m})}}ds_{m} \overline{s} \int\limits_{\partial{(U^{\prime}_{\sigma}\cap \mathbb{C}_{m})}} S_{L}^{-1}(p,T) g_{s}(p)dp_{m}\\
&=\frac{1}{4\pi^{2}}\int\limits_{\partial{(U_{\sigma}\cap \mathbb{C}_{m})}}ds_{m} \int\limits_{\partial{(U^{\prime}_{\sigma}\cap \mathbb{C}_{m})}} \big[\overline{s}S_{L}^{-1}(p,T) - S_{L}^{-1}(p,T)p\big]g_{s}(p) dp_{m} \\
&= \frac{1}{2\pi}\int\limits_{\partial{(U^{\prime}_{\sigma}\cap \mathbb{C}_{m})}} \Big(\frac{1}{2\pi}\int\limits_{\partial{(U_{\sigma}\cap \mathbb{C}_{m})}}ds_{m}\big[\overline{s}S_{L}^{-1}(p,T) - S_{L}^{-1}(p,T)p\big]g_{s}(p) \Big)\; dp_{m}\\
&= \frac{1}{2\pi}\int\limits_{\partial{(U^{\prime}_{\sigma}\cap \mathbb{C}_{m})}}S_{L}^{-1}(p, T)\; dp_{m}, \; \; \;  \text{by Equation}\  (\ref{Equation: Finally})\\
&= P_{\sigma}.
\end{align*}
Further, by Remark \ref{Remark: adjoint} and Equation (\ref{Equation: Psigmaforprime}), the operator $P_{\sigma}$ is self-adjoint. Similarly the result holds true for $P_{\tau}$. Therefore, $P_{\sigma}$ and $P_{\tau}$ are orthogonal projections in $\mathcal{H}$, we call them as {\it Riesz projections}.

	\vspace{0.3cm}
	
	\noindent {\bf Step II:} Let $\mathcal{M}_{\sigma} := R(P_{\sigma}) $ and $\mathcal{M}_{\tau} = R(P_{\tau})$. Then $\mathcal{H} = \mathcal{M}_{\sigma} \oplus \mathcal{M}_{\tau}$.
	\vspace{0.2cm}
	
Let us define $U:= U_{\sigma} \cup U_{\tau}$. Then $U$ is an axially symmetric $s$-domain containing $\sigma_{S}(T)$ such that $\partial{(U\cap \mathbb{C}_{m})}$ is the union of a finite number of continuously differentiable Jordan curves. Equivalently, $U$ is a $T$-admissible domain. By the quaternionic functional calculus, we have
\begin{align*}
    P_{\sigma} + P_{\tau} 
    &= \int\limits_{\partial{(U_{\sigma}\cap \mathbb{C}_{m})}} d_{m}\ S_{R}^{-1}(s, T) +\int\limits_{\partial{(U_{\tau}\cap \mathbb{C}_{m})}} ds_{m}\ S_{R}^{-1}(s, T) \\
    &= \int\limits_{\partial{(U\cap \mathbb{C}_{m})}} ds_{m}\ S_{R}^{-1}(s, T) \\
    &= I,
\end{align*}
	where $I \in \mathcal{B}(\mathcal{H})$ is the identity operator. It follows that
	\begin{equation*}
	P_{\sigma}\cdot P_{\tau} = P_{\sigma}(I-P_{\sigma}) = P_{\sigma} - P_{\sigma}^{2} = 0.
	\end{equation*} 
	If $x \in \mathcal{H}$, then $x$ is uniquely expressed as, $x = P_{\sigma}(x) + P_{\tau}(x)$. This implies that  
	\begin{equation*}
	    \mathcal{H} = \mathcal{M}_{\sigma} \oplus \mathcal{M}_{\tau}.
	\end{equation*} 	
	\noindent {\bf Step III:} We prove that $M_{\sigma}$ and $M_{\tau}$ are invariant subspaces of $T$. 
	\vspace{0.3cm}
	
	 It is enough to show that both $P_{\sigma}$ and $P_{\tau}$ commute with $T$.  For every $x, y \in \mathcal{H}$, we compute that
	\begin{align*}
	\big\langle x, TP_{\sigma}(y)\big\rangle = \big\langle T^{\ast}(x), P_{\sigma}(y)\big\rangle &= \Big\langle T^{\ast}x,\; \Big(\frac{1}{2\pi}\int\limits_{\partial{(U_{\sigma}^{\prime}\cap \mathbb{C}_{m})}} S_{L}^{-1}(p,T)\ dp_{m} \Big) y\Big\rangle\\
	&=\frac{1}{2\pi}\int\limits_{\partial{(U_{\sigma}^{\prime}\cap \mathbb{C}_{m})}}\big\langle T^{\ast}x,\;  S_{L}^{-1}(p, T)y \big\rangle\;  dp_{m}\\
	&= \frac{1}{2\pi}\int\limits_{\partial{(U_{\sigma}^{\prime}\cap \mathbb{C}_{m})}}\big\langle x,\; TS_{L}^{-1}(p, T)y \big\rangle \; dp_{m}\\
	&=\Big\langle x,\; \Big(\frac{1}{2\pi}\int\limits_{\partial{(U_{\sigma}^{\prime}\cap \mathbb{C}_{m})}} TS_{L}^{-1}(p,T)\ dp_{m} \Big) y\Big\rangle.
	\end{align*}
This implies that 
\begin{align*}
TP_{\sigma} &= \int\limits_{\partial{(U_{\sigma}^{\prime}\cap \mathbb{C}_{m})}} TS_{L}^{-1}(p, T) \; dp_{m}, \;\; \text{where}\; dp_{m}= -dp\cdot m\\
&= \int\limits_{\partial{(U_{\sigma}^{\prime}\cap \mathbb{C}_{m})}} \big(S_{L}^{-1}(p, T)p-I\big) \; dp_{m} ,\; \; \text{by Equation}\; (\ref{Equation: S_{L}})\\
&= \int\limits_{\partial{(U_{\sigma}^{\prime}\cap \mathbb{C}_{m})}} S_{L}^{-1}(p,T)dp_{m}\; p, \; \; \; \text{since}\; \int\limits_{\partial{(U_{\sigma}\cap \mathbb{C}_{m})}} dp_{m} = 0.
\end{align*}
Similarly, we compute $P_{\sigma}T$ as, 
	\begin{align*}
	\big\langle x, P_{\sigma}(Ty)\big\rangle  &= \Big\langle x,\; \Big(\frac{1}{2\pi}\int\limits_{\partial{(U_{\sigma}\cap \mathbb{C}_{m})}} ds_{m}\; S_{R}^{-1}(s,T) \Big) (Ty)\Big\rangle\\
	&=\frac{1}{2\pi}\int\limits_{\partial{(U_{\sigma}\cap \mathbb{C}_{m})}}ds_{m}\; \big\langle x,\;  S_{R}^{-1}(s, T)Ty \big\rangle\\
	&= \frac{1}{2\pi}\int\limits_{\partial{(U_{\sigma}\cap \mathbb{C}_{m})}}ds_{m}\; \big\langle x,\; S_{R}^{-1}(s, T)Ty \big\rangle \\
	&=\Big\langle x,\; \Big(\frac{1}{2\pi}\int\limits_{\partial{(U_{\sigma}\cap \mathbb{C}_{m})}} ds_{m}\; S_{R}^{-1}(s,T)T  \Big) y\Big\rangle.
	\end{align*}
	Thus

	\begin{align*}
	P_{\sigma}T &= \int\limits_{\partial{(U_{\sigma}\cap \mathbb{C}_{m})}} ds_{m}\; S_{R}^{-1}(s, T)T, \; \; \text{where}\; ds_{m} = -ds\cdot m\\
	&= \int\limits_{\partial{(U_{\sigma}\cap \mathbb{C}_{m})}} ds_{m}\; \big(s S_{R}^{-1}(s, T)- I\big), \; \; \; \text{by Equation}\; (\ref{Equation: S_{R}})\\
	&= \int\limits_{\partial{(U_{\sigma}\cap \mathbb{C}_{m})}} s\;  ds_{m}\; S_{R}^{-1}(s, T),\; \; \; \text{since}\; \int\limits_{\partial{(U_{\sigma}\cap \mathbb{C}_{m})}} ds_{m} = 0.
	\end{align*}
	This shows that  $TP_{\sigma} = P_{\sigma}T$. Therefore,
	\begin{equation*}
	TP_{\tau} = T(I-P_{\sigma}) = (I-P_{\sigma})T = P_{\tau}T.
	\end{equation*} 
	As a result, we conclude that both $\mathcal{M}_{\sigma}$ and $\mathcal{M}_{\tau}$ are invariant subspaces of $T$.

	\vspace{0.3cm}
	\noindent {\bf Step IV:}  Finally, we show that $\sigma= \sigma_{S}(T|_{\mathcal{M}_{\sigma}})$ and $\tau = \sigma_{S}(T|_{\mathcal{M}_{\tau}})$. 
	\vspace{0.2cm}
	
	Suppose that $q \notin \sigma$, then $[q] \notin {\sigma}$ since $\sigma $ is axially symmetric. With out loss of generality, we assume that there is an axially symmetric $s$-domain   $U_{\sigma}$ containing $\sigma$ such that   $\partial{(U_{\sigma}\cap \mathbb{C}_{m})}$ is the union of a finite number of continuously differentiable Jordan curves for every $m \in \mathbb{S}$. Let us fix $m \in \mathbb{S}$. Define the operator 
	\begin{equation*}
	\mathcal{Q}_{\sigma}^{(q)}: = \frac{1}{2\pi} \int\limits_{\partial{(U_{\sigma}\cap \mathbb{C}_{m})}} S_{L}^{-1}(t, T)\; dt_{m} \big(t^{2}- 2 re(q)t+|q|^{2}\big)^{-1},
	\end{equation*}
	where $dt_{m} = -dt\cdot m$. We claim that $\mathcal{M}_{\sigma}$ is invariant subspace of $\mathcal{Q}_{\sigma}^{(q)}$. For any $p \in \mathbb{C}_{m}$, if we define a map $\xi_{p}(t) = \big(t^{2}- 2 re(q)t+|q|^{2}\big)^{-1} S_{L}^{-1}(p, t)$, for all $t \in U_{\sigma}$. It is clear that 
	\begin{equation*}
	    \big(t^{2}- 2 re(q)t+|q|^{2}\big)^{-1} \in \mathbb{C}_{m}, \; \; \text{whenever}\;\; t \in \mathbb{C}_{m} 
	\end{equation*} 
	and $S_{L}^{-1}(p,t)$ is a left $s$-regular function in variable $t$. Thus $\xi_{p} \in R^{L}(U_{\sigma})$ by  \cite[Proposition 4.11.5]{Colombo}. So, by the quaternionic functional calculus, we deduce that
	\begin{equation} \label{Equation: xip}
	    \xi_{p}(T) = \mathcal{Q}^{(q)}_{\sigma} S_{L}^{-1}(p, T) =   \frac{1}{2\pi} \int\limits_{\partial(U_{\sigma}\cap \mathbb{C}_{m})} S_{L}^{-1}(t, T)\ dt_{m} \big(t^{2}- 2 re(q)t+|q|^{2}\big)^{-1} S_{L}^{-1}(p, t).
	\end{equation}
	This implies the following:
	\begin{align*}
	    \mathcal{Q}_{\sigma}^{(q)}P_{\sigma} &= \frac{1}{2\pi} \int\limits_{\partial{(U_{\sigma}^{\prime}\cap \mathbb{C}_{m})}} \mathcal{Q}^{(q)}_{\sigma} S_{L}^{-1}(p, T) dp_{m}\\
	    &= \frac{1}{4\pi^{2}} \int\limits_{\partial{(U_{\sigma}^{\prime}\cap \mathbb{C}_{m})}} \int\limits_{\partial(U_{\sigma}\cap \mathbb{C}_{m})} S_{L}^{-1}(t, T)\ dt_{m} \big(t^{2}- 2 re(q)t+|q|^{2}\big)^{-1} S_{L}^{-1}(p, t) dp_{m},\\
	    &\hspace{9cm}\; \text{by Equation}\ (\ref{Equation: xip})\\
	    &= \frac{1}{2\pi} \int\limits_{\partial{(U\cap \mathbb{C}_{m})}} S_{L}^{-1}(t, T)\ dt_{m} \big(t^{2}- 2 re(q)t+|q|^{2}\big)^{-1} \Big(\frac{1}{2\pi} \int\limits_{\partial{(U_{\sigma}^{\prime}\cap \mathbb{C}_{m})}}S_{L}^{-1}(p, t) dp_{m}\Big)\\
	    &= \frac{1}{2\pi} \int\limits_{\partial{(U\cap \mathbb{C}_{m})}} S_{L}^{-1}(t, T)\ dt_{m} \big(t^{2}- 2 re(q)t+|q|^{2}\big)^{-1}, \\
	    & \hspace{5cm}\;\text{since}\; \frac{1}{2\pi} \int\limits_{\partial{(U_{\sigma}\cap \mathbb{C}_{m})}} S_{L}^{-1}(s, t)\; ds_{m} = 1\\
	    &= \mathcal{Q}_{\sigma}^{(q)}.
	\end{align*}
	 Similarly, $P_{\sigma}\mathcal{Q}_{\sigma}^{(q)} = \mathcal{Q}_{\sigma}^{(q)}$. This means $\mathcal{Q}_{\sigma}^{(q)}|_{\mathcal{M}_{\sigma}} \in \mathcal{B}(\mathcal{M}_{\sigma})$. Next, we show that $\Delta_{q}(T)|_{\mathcal{M}_{\sigma}}$ is invertible. For this, let us define $\varphi_{q}(t) = t^{2}-2re(q)t+|q|^{2}$. Then $\varphi_{q}$ is  locally $s$-regular function on $\sigma_{S}(T)$. Moreover, by following similar arguments as in Step III, we express that
	 \begin{equation}\label{Equation: Deltap}
	     \Delta_{q}(T)|_{\mathcal{M}_{\sigma}} P_{\sigma} = \varphi_{q}(T) P_{\sigma} =  \frac{1}{2\pi} \int\limits_{\partial{(U_{\sigma}^{\prime}\cap \mathbb{C}_{m})}} S_{L}^{-1}(p, T)\ dp_{m} (p^{2}- 2 re(q)p+|q|^{2}),
	 \end{equation}
	where $dp_{m} = -dp\cdot m$.  Also we know that $\varphi_{q}(p)S_{L}^{-1}(t,p)$ is left $s$-regular function in the variable $p$, for every $t$. By \cite[Proposition 4.11.5]{Colombo}, it follows that
	\begin{equation}\label{Equation: Also}
	    \Delta_{q}(T)|_{\mathcal{M}_{\sigma}}S_{L}^{-1}(t,T) = \frac{1}{2\pi} \int\limits_{\partial(U_{\sigma}^{\prime}\cap \mathbb{C}_{m})} S_{L}^{-1}(p,T)\ dp_{m} \ \varphi_{q}(p) S_{L}^{-1}(t,p).
	\end{equation}
	From Equations (\ref{Equation: Deltap}), (\ref{Equation: Also}) we compute $\Delta_{q}(T)|_{\mathcal{M}_{\sigma}}Q_{\sigma}^{(q)}$ as follows:
	\begin{align*}
	&\Delta_{q}(T)|_{\mathcal{M}_{\sigma}} \mathcal{Q}_{\sigma}^{(q)} \\
	&= \frac{1}{2\pi} \int\limits_{\partial{(U_{\sigma}\cap \mathbb{C}_{m})}} \Delta_{q}(T)|_{\mathcal{M}_{\sigma}}S_{L}^{-1}(t,T)\; dt_{m}\; \big(t^{2} - 2re(q)t+|q|^{2}\big)^{-1}\\
	&=\frac{1}{4\pi^{2}} \int\limits_{\partial{(U_{\sigma}\cap \mathbb{C}_{m})}} \int\limits_{\partial{(U_{\sigma}^{\prime}\cap \mathbb{C}_{m})}} S_{L}^{-1}(p,T)\;  dp_{m}\; \varphi_{q}(p) S_{L}^{-1}(t,p)\; dt_{m} \big(t^{2} - 2re(q)t)+|q|^{2}\big)^{-1}\\
	&= \frac{1}{2\pi} \int\limits_{\partial(U^{\prime}_{\sigma}\cap \mathbb{C}_{m})} S_{L}^{-1}(p,T)\;  dp_{m}\; \varphi_{q}(p) \Big ( \frac{1}{2\pi}\int\limits_{\partial(U_{\sigma \cap \mathbb{C}_{m}})} S_{L}^{-1}(t,p)\; dt_{m} \big(t^{2} - 2re(q)t)+|q|^{2}\big)^{-1} \Big)\\
	&= \frac{1}{2\pi} \int\limits_{\partial(U_{\sigma}^{\prime} \cap \mathbb{C}_{m})} S_{L}^{-1}(p,T)\;  dp_{m}\; \varphi_{q}(p) \big(p^{2} - 2re(q)p)+|q|^{2}\big)^{-1}.
	\end{align*}
	Since $\varphi_{q}(p)  \big(p^{2} - 2re(q)p+|q|^{2}\big)^{-1} = 1$, it follows that $\Delta_{q}(T)|_{\mathcal{M}_{\sigma}} \mathcal{Q}_{\sigma}^{(q)} = P_{\sigma}$. 
Similarly, one can show that $\mathcal{Q}_{\sigma}^{(q)}\Delta_{q}(T)|_{\mathcal{M}_{\sigma}} = P_{\sigma}$. In other words, for every $x \in \mathcal{M}_{\sigma}$, we see that 
	\begin{equation*}
	 \Delta_{q}(T|_{\mathcal{M}_{\sigma}}) \mathcal{Q}_{\sigma}^{(q)}x =\mathcal{Q}_{\sigma}^{(q)} \Delta_{p}(T|_{\mathcal{M}_{\sigma}}) x = x, \; \text{for all}\; x \in \mathcal{M}_{\sigma}.
	\end{equation*} 
	Thus $\Delta_{q}(T|_{\mathcal{M}_{\sigma}})$ is invertible and hence $q \in \rho_{S}(T|_{\mathcal{M_{\sigma}}})$. It follows that  
	\begin{equation}\label{Equation: contained}
	    \sigma_{S}(T|_{\mathcal{M}_{\sigma}}) \subseteq \sigma.
	\end{equation} 
	By the similar arguments, we achieve that
	\begin{equation}\label{Equation: contained2}
	    \sigma_{S}(T|_{\mathcal{M}_{\tau}}) \subseteq \tau.
	\end{equation}
	
	Now we prove reverse inclusions. Suppose that $q \notin \sigma_{S}(T|_{M_{\sigma}})\cup \sigma_{S}(T|_{M_{\tau}})$. Then both operators $\Delta_{q}(T|_{\mathcal{M}_{\sigma}})$ and $\Delta_{q}(T|_{\mathcal{M}_{\tau}})$ are invertible. Hence $\Delta_{q}(T) \in \mathcal{B}(\mathcal{H})$ is invertible since $\mathcal{H} = \mathcal{M}_{\sigma} \oplus \mathcal{M}_{\tau}$. Equivalently,  $q \notin \sigma_{S}(T)$. This shows that
	 \begin{equation*}
	\sigma_{S}(T) \subseteq \sigma_{S}(T|_{M_{\sigma}})\cup \sigma_{S}(T|_{M_{\tau}}) \subseteq \sigma \cup \tau= \sigma_{S}(T).
	\end{equation*}
	Therefore by Equation (\ref{Equation: contained}), (\ref{Equation: contained2}) and using the fact that $\sigma$ and $\tau$ are disjoint, we conclude that
	\begin{equation*}
	\sigma_{S}(T|_{M_{\sigma}})= \sigma\; \text{and}\; \sigma_{S}(T|_{M_{\tau}})=\tau. 
	\end{equation*}
	Hence the result.
\end{proof}


\begin{cor} Let $T \in \mathcal{B}(\mathcal{H})$. If the spherical spectrum $\sigma_{S}(T)$ is disconnected by a pair of disjoint nonempty axially symmetric closed subsets, then  $T$ is strongly reducible. 
\end{cor}
\begin{proof}
	From the hypothesis, assume that there is a pair $\{\sigma, \tau\}$  of disjoint nonempty axially symmetric closed subsets of $\sigma_{S}(T)$ satisfying,
	\begin{equation*}
	\sigma_{S}(T) = \sigma \cup \tau. 
	\end{equation*} 
	 Then by Theorem \ref{Theorem: Riesztheorem}, there exist a pair of nontrivial mutually orthogonal invariant subspaces $M_{\sigma}$ and $M_{\tau}$ of $T$ such that 
	 \begin{equation*}
	 \sigma_{S}(T|_{M_{\sigma}})= \sigma\; \; \text{and}\;\; \sigma_{S}(T|_{M_{\tau}}) = \tau.
	 \end{equation*} 
	 Equivalently, $T$ commutes with the corresponding projections $P_{\sigma}$ and $P_{\tau}$ as shown in  step III of Theorem \ref{Theorem: Riesztheorem}.  This implies that $T$ is strongly reducible. 
\end{proof}
\subsection*{Acknowledgments}
The author thanks NBHM (National Board for Higher Mathematics, India) for financial support with ref No. 0204/66/2017/R\&D-II/15350, and Indian Statistical Institute Bangalore for providing necessary facilities to carry out this work. 

The author would like to express his sincere gratitude to Prof. B.V. Rajarama Bhat for helpful suggestions and constant support.



\end{document}